\documentclass[letterpaper,12pt]{amsart}
\advance\oddsidemargin by-0.5in
\advance\evensidemargin by-0.5in
\advance\textwidth by 1in
\usepackage{setspace,graphicx,color,amsmath,amssymb,mathptmx,amsthm}
\usepackage[bookmarks=false]{hyperref}
\newcommand{\E}[2]{#1^{(#2)}}
\newcommand{\gateaux}{G\^ateaux}
\newcommand{\frechet}{Fr\'echet}
\newcommand{\real}{\mathbb{R}}
\newcommand{\R}{\mathbb{R}}
\newcommand{\mT}{\mathbb{T}}
\newcommand{\G}{O}
\newcommand{\Z}{\mathbb{Z}}
\newcommand{\e}{\varepsilon}
\newcommand{\ball}[2]{B_{#2}(#1)}
\newcommand{\ballcl}[2]{\overline{B}_{#2}(#1)}
\newcommand{\T}{T}
\newcommand{\limx}{x_\infty}
\newcommand{\lime}{e_\infty}
\renewcommand{\phi}{\varphi}
\newtheorem{theorem}[subsection]{Theorem}

\newtheorem{lemma}[subsection]{Lemma}
\newtheorem*{lemma*}{Lemma}
\newtheorem{definition}[subsection]{Definition}

\theoremstyle{definition}
\newtheorem{remark}[subsection]{Remark}
\newtheorem*{remark*}{Remark}
\numberwithin{equation}{section}

\newcommand{\bib}{
}
\begin{document}
\advance\baselineskip by3pt
\title{A universal differentiability set in Banach spaces with separable dual}
\author{Michael Dor\'e}
\address{School of Mathematics, University of Birmingham, Edgbaston, Birmingham B15 2TT, UK}
\email{M.J.Dore@bham.ac.uk}
\author{Olga Maleva}
\address{School of Mathematics, University of Birmingham, Edgbaston, Birmingham B15 2TT, UK}
\email{O.Maleva@bham.ac.uk}
\thanks{Michael Dor\'e acknowledges the support of the UK EPSRC PhD+ programme. Olga Maleva acknowledges the support of the EPSRC grant
EP/H43004/1.}
\begin{abstract}
We show that any non-zero Banach space with a separable dual
contains a totally disconnected, closed and bounded subset $S$ of Hausdorff dimension $1$
such that every Lipschitz function on the space is
Fr\'echet differentiable somewhere in $S$.
\end{abstract}

\maketitle
\pagestyle{myheadings}
\markboth{A UNIVERSAL DIFFERENTIABILITY SET}{MICHAEL DOR\'E and OLGA MALEVA}

\section{Introduction}
\label{intro}\label{sec.intro}
It is well known that
there are quite strong results ensuring the existence of points of differentiability of Lipschitz functions defined on
finite and infinite dimensional Banach spaces.
Rademacher's theorem implies that real-valued Lipschitz functions on finite dimensional spaces are differentiable almost everywhere in the sense of Lebesgue measure.  For the infinite dimensional case, Preiss shows in \cite[Theorem 2.5]{P} that
every real-valued Lipschitz function defined on an Asplund\footnote{This is best possible as
any non-Asplund space has an equivalent norm
--- which of course is a Lipschitz function --- that is nowhere \frechet{}
differentiable; see \cite{Asp,BL}.} space
is \frechet{} differentiable at a dense set of points.

A natural question then arises as to whether every ``small'' set $S$
in a finite dimensional or infinite dimensional Asplund space $Y$
gives rise to a real-valued Lipschitz function on $Y$ not differentiable
at any point of $S$.
Let us call a subset $E$ of the space $Y$ a universal differentiability
set if for every Lipschitz function $f:Y\to\real$, there exists $y\in E$ such that
$f$ is \frechet{} differentiable at $y$.

In this paper we show that for non-zero separable Asplund spaces $Y$, there are always
``small'' subsets with the universal differentiability property, in the sense that the Hausdorff dimension of the closure of the set can be taken equal to $1$. Hence, as we may also take the set to be bounded, in the case in which $Y$ is a finite dimensional space of dimension at least $2$ we recover the fact that a universal differentiability set may be taken to be compact and with Lebesgue measure zero, a fact first proved by the authors in~\cite{DM}.

In the case $\dim Y=1$ it is easy to show that every
Lebesgue null subset of $\real$ is not a universal differentiability set;
see \cite{Z} and \cite{FT}.  Note also that a separable Asplund space is simply a Banach space with a separable dual.
For non-separable spaces $Y$, any set $S$ of finite Hausdorff dimension
is contained in a separable subspace
$Y'\subseteq Y$; therefore the distance function
$y \mapsto \text{dist}(y,Y')$ is Lipschitz and nowhere differentiable on $S$.

We note here that for Lipschitz mappings whose codomain has dimension $2$ or above, there are  many
open questions. For example, while  Rademacher's  Theorem
still guarantees
that for every $n\ge m\ge2$ the set of points where a Lipschitz mapping
$f:\real^n\to\real^m$ is not differentiable has Lebesgue measure zero,
the answer to the question of
whether there are Lebesgue null sets in $\R^n$ containing a differentiability point of every Lipschitz $f \colon \R^n \to \R^m$
is known only for $m=2$. The answer for $n=m=2$ is negative; see \cite{ACP}. The case $n>m=2$ is a topic of a forthcoming
paper \cite{DM4} where the authors, building on methods
developed in \cite{LPT} in their
study of differentiability problems in infinite dimensional Banach spaces,
construct null universal differentiability sets for planar valued Lipschitz
functions.

No similar positive results are known in the case in which the dimension of the codomain is at least $3$. However, a partial result was obtained in \cite{PH} where it is proved that the
union $H$ of all ``rational hyperplanes'' in $\real^n$ has the property that
for every $\e>0$ and every Lipschitz mapping $f:\real^n\to\real^{n-1}$
there is a point in $H$ where the function $f$ is $\e$-\frechet{}
differentiable. Unfortunately, this is a weaker notion, and the existence of points
of $\e$-\frechet{} differentiability does not imply the existence of points
of full differentiability. See also \cite{JLPS,LP}, in which the notion of
$\e$-\frechet{} differentiability is studied with the emphasis on the infinite dimensional case.

It follows from the work of Preiss in \cite{P}
that Lebesgue null universal differentiability sets exist in any Euclidean space of dimension at least $2$.
However there is a drawback
in the construction by Preiss: any set $S$ covered by \cite[Theorem 6.4]{P}
is dense in the whole space, and simple refinements of the same approach are only capable of constructing universal differentiability sets that are still dense in some non-empty open set.
This can be explained as follows. The proof in~\cite{P} makes essential use of the following sufficient condition for $S$ to be a universal differentiability set:
$S$ is $G_{\delta}$ and for every $x \in S$ and $\e>0$,
there is a $\delta$-neighbourhood $N$ of $x$, for some $\delta = \delta(\e,x) > 0$, such that for every line segment $I \subseteq N$,
the set contains a large portion
of a path that approximates $I$ to within $\e|I|$. Fixing $\e = 1/2$ say, a simple application of the Baire category theorem shows that one can choose $\delta(1/2,x)$ uniformly
for $x$ belonging to some non-empty open $U$. It quickly follows that $S$ itself is dense in $U$.
See also \cite[Introduction]{DM} for a discussion of this.

In \cite{DM2} we improve the result of \cite[Corollary 6.5]{P}
by constructing,
in every finite dimensional space,
a compact universal differentiability set that has  Hausdorff dimension $1$.

The main result of the present paper is that
every non-zero Banach space with separable dual
contains a closed and bounded universal differentiability
set of Hausdorff dimension $1$; see Theorem~\ref{th.main-part1},
Remark~\ref{rem.wedge},
Lemma~\ref{lem.main-cl}
and Theorem~\ref{th.UDtot}.
The dimension $1$ here is optimal; see Lemma~\ref{lem.hausd}.
The universal differentiability set need not contain any non-constant
continuous curves; in Theorem~\ref{th.UDtot}
we show that this set may in fact be chosen to be totally disconnected.
In the case in which $Y$ is a finite dimensional space, this result implies
the earlier result of \cite{DM2}.
Note that compact subsets of infinite dimensional spaces cannot have the
universal differentiability property; indeed if $S \subseteq Y$ is compact then
one may even construct a Lipschitz
convex function $f \colon Y\to \R$ not Fr\'echet differentiable on $S$, for example
$$
f(y) := \text{dist}(y,{\text{convex hull}(S)}).
$$
See also remark after Lemma~\ref{perturbcompacts}.

The proof of  Lemma~\ref{lem.main-cl}
is based on Theorems~\ref{th.main-part1} and~\ref{th.passtoclosed}, which rely on
Sections~\ref{Differentiability},
\ref{Optimisation}
and~\ref{Set_theory}.
Section~\ref{Set_theory} gives details of the construction of the set.
Section~\ref{Optimisation} explains the procedure for finding the point
with almost locally maximal directional derivative. Finally, Section~\ref{Differentiability} proves any such point is a point of
\frechet{} differentiability.

Assume we have a closed set $S$ and that we aim to prove $S$ has the universal
differentiability property. We describe the details of the
construction of $S$ below; at the
moment we just say $S$ is going to be defined using a Souslin-like operation
on a family of closed ``tubes'', that is
closed neighbourhoods of particular line segments.
Consider an arbitrary
Lipschitz function $f:Y\to\real$; we would like to show $f$ is \frechet{} differentiable
at some point of $S$.
The strategy is to, in some sense, almost locally maximise the
directional derivative of $f$;
this is done in Theorem~\ref{th.incr},
 from within the constructed family of tubes. We then use the Differentiability Lemma~\ref{lem.diff},
which gives a sufficient condition
for the Lipschitz function to be \frechet{} differentiable at a point
where it has such an `almost locally maximal' directional derivative.

In Section~\ref{Differentiability} we prove that if a Lipschitz function
$f$ has a
directional derivative $L$ at some point $y\in S$, and this derivative
is almost locally maximal in the sense that for every $\e$,
every directional derivative at \textit{any} nearby point from $S$ does not
exceed $L+\e$, then the Lipschitz function is in fact
\frechet{} differentiable at the original point
and the gradient is in the direction $e$ of the almost locally maximal
directional  derivative.
The word \textit{any} in the latter sentence needs in fact to be replaced by a special condition \eqref{eq.condi}; see Lemma~\ref{lem.max} and Lemma~\ref{lem.diff}.
The proof is then based on the idea that, assuming non-\frechet{} differentiability, we can find a wedge --- that is a specially chosen union of two line segments --- in which the growth of the function contradicts the mean value theorem and the local maximality assumption.

In Section~\ref{Optimisation} we show how to find such point with
`almost locally maximal' directional derivative.
The idea behind the proof is to take a sequence of pairs $(y_n,e_n)$ with the directional derivative $g'(y_n,e_n)$ being very close
to the supremum over all directional derivatives $g'(z,u)$
with $z$ close to $y_{n-1}$ and $(z,u)$ satisfying certain additional constraints -
see Definition~\ref{defnweight} and inequality~\eqref{eq.defxnen} -
and to argue that the sequence
$(y_n,e_n)$ converges to a point-direction pair $(y,e)$
with the desired
almost locally maximal directional derivative.

The optimisation method used in the present paper develops ideas from \cite{P}
and  \cite{DM}. The new idea that we use in this paper is that
instead of looking at points $y\in Y$, we define a bundle $X$ over $Y$, where $X$ is a complete topological space and $\pi \colon X \to Y$ is a continuous mapping, and locally maximise the directional
derivative $f'(\pi x, e)$ over $x \in X$.
This ensures that during the optimisation iterative
procedure we are not thrown to the boundary of the set; if $\pi(X) \subseteq S$ then we are guaranteed that the point we obtain lies inside $S$.

Another key aspect of the proof of our result is the new set theoretic construction; see Theorem~\ref{th.passtoclosed} and Section~\ref{Set_theory}.
First of all, we need to remark that the limit point to be obtained as a
result
of optimisation procedure must not be a porosity point of the set --- see the next section for the definition and reasons. We achieve this by
constructing a set in which, for every point $x$ and every $\e>0$, sufficiently
small $\delta$-neighbourhoods of $x$ contain  an $\e \delta$-dense set of line segments. The set is defined as an
intersection of a countable collection $(J_k(\lambda))_{k \geq 1}$ of closed sets. Each $J_k(\lambda)$ is in its turn a countable union of ``tubes'',  which are
closed neighbourhoods of  particular line segments.
The construction of $J_k(\lambda)$ is inductive: around every tube in $J_l(\lambda)$ with $l < k$ we add a fine collection of tubes to $J_{k}(\lambda)$
and replace the original tube with a more narrow tube around the same
line segment.

As we are aiming for a final set of Hausdorff dimension $1$, we need to ensure the widths of the tubes in $J_k(\lambda)$ tend to $0$ as $k \to \infty$. More precisely,
we fix upfront a $G_\delta$ set $\G$
of Hausdorff dimension $1$ containing a dense set of straight line segments,
 and a nested collection of open sets $\G_k$ with intersection $\G$.
By constructing $J_k(\lambda) \subseteq \G_k$ we thereby ensure that
$\T_\lambda=\bigcap_{k \geq 1} J_k(\lambda)$
has Hausdorff dimension at most $1$, as required. As $J_k(\lambda)$ are closed sets, so then is $T_{\lambda}$.

The parameter $\lambda\in(0,1)$ is used to
change the widths of all tubes involved in tube sets
$J_k(\lambda)$ proportionally,
multiplying by $\lambda$. We then establish that if $\lambda_1<\lambda_2$ are fixed and
we pick an arbitrary point $y\in\T_{\lambda_1}$, then for each $\e>0$
every sufficiently small $\delta$-neighbourhood of $y$ has an
$\e\delta$-dense set of line segments that are fully
inside $T_{\lambda_2}$.
In order to achieve this we first find the level $N$ after which,
in the construction of tube sets  $J_k(\lambda)$ we were
choosing new tubes with density finer than
$\e$ multiplied by the width of the tube on the previous level.
Choose $\delta$ to be smaller than the width of a tube on the level $N$
and set $n\ge N$ to be the ``critical'' level on which the width of the
tube containing point $x$ multiplied by $\lambda_2-\lambda_1$
for the first time becomes less than $\delta$. Then
the whole $\delta$-neighbourhood of $x$ is
guaranteed to be inside the tube sets $J_m(\lambda_2)$, with $m\le n-1$.
For $m\ge n+1$ we find that the new tubes go $\e w_m$-densely around $x$,
where $w_m$ is the width on the tube on level $m$. Since
$w_m\le\e w_{m-1}\le\e\delta$ by construction, we  find
many tubes $\e\delta$-close to $x$ on those subsequent levels. The problem that
remains is that on the level $n$ itself we might not find an appropriate tube at
all! We overcome this obstacle by slightly modifying the definition of
$J_k(\lambda)$ and taking it to be the union of tubes on a number of
levels so that the ``one level shift'' does not take us outside the
tube set $J_k(\lambda)$.

\begin{figure}[h]
\label{fig}
\begin{center}
\input{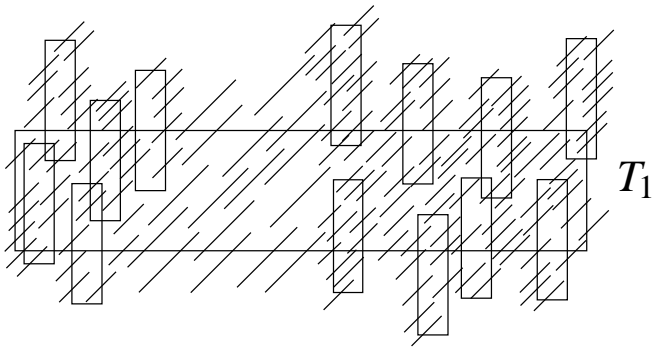}
\caption{We show here a horizontal tube $T_1$ of level $1$, vertical tubes
of level $2$ that approximate points from $T_1$ and ``diagonal'' tubes of level $3$ that
approximate points from tubes of level $2$ and points from $T_1$.}
\end{center}
\end{figure}

There is extra problem in the infinite dimensional case however.
Given a tube $T$ of width $w$ in one of the tube sets $J_k$,
in order to ``kill'' its porosity points and ensure sufficiently many line segments in the final set, we add tubes $w/N_k$-densely to $J_{k+1}$,
where $N_k \to \infty$. The problem that
immediately arises in the infinite dimensional case is that there is
no ``minimal'' width among all tubes from $J_k$: since the Banach spaces we are working over are not
locally compact, each collection of tubes $J_k$ will have to be infinite, so that the infimum of the widths in $J_k$ may be zero.
Therefore we must add such approximating tubes only locally, in a small neighbourhood of each tube from $J_k$. This forces the length of tubes
close to any fixed point $x\in\T_\lambda$ to shrink rapidly, and therefore
the point $x$
will not have a ``safe'' neighbourhood $\ball{x}{r}$ in which the set hits
every ball $\ball{y}{c\|x-y\|}$,
i.e.\ we again get porosity at $x$.

To overcome this, a new approach is required; see Definition~\ref{defMinf} and the proof of Lemma~\ref{thehardwork}. In brief, on constructing level $k+1$ approximation of tube $T$ from $J_k$,
we re-visit tubes constructed on each previous
levels, $1\le l\le k$, that form a sequence of ancestors of $T$. We approximate
each tube thus re-visited to level $k+1$ and include all new tubes in
$J_{k+1}$ --- see Figure~\ref{fig}. Approximations of lower level tubes to level $k+1$
allows us to include longer tubes in $J_{k+1}$.
This makes it possible to find, for each $x$ and $\e>0$, the critical value
$\delta_1>0$ such that $\e\delta$-close to $x$ there are  line
segments of length $\delta$, for every $\delta\in (0,\delta_1)$.
As explained in the beginning of this section,
this property turns out to be sufficient for the set to have the universal
differentiability property.

Theorem~\ref{th.passtoclosed} is stated using more general terms than line segments and tubes; we prove the statement for a general class $(K_r)_{r \in R}$ of compact subsets of an arbitrary metric space $(Y,d)$.

Finally, to get a totally disconnected universal differentiability set, we
need to get rid of all these straight line segments that we have included
in order to be able to prove the differentiability property
inside the set $T_\lambda$ defined above.
For this, we intersect $T_\lambda$ with a union of parallel hyperplanes obtained as a preimage of a totally disconnected subset of $\real$ under a continuous linear functional.
To have this intersection totally disconnected it is enough to ensure that the containing $G_\delta$ set has this intersection totally disconnected.
To show that the intersection of $T_\lambda$ with the union of hyperplanes has the universal differentiability property we prove that for every Lipschitz function, its differentiability points inside $T_\lambda$ form a very dense subset, and then choose the hyperplanes densely enough.
See Theorems~\ref{th.main-part1} and~\ref{th.UDtot} for details.

\section{Definitions and notations}
\label{sec.def}
In this paper we shall be working with real valued functions
defined on a real Banach space $Y$ with separable dual. If a function
$f:E\to\real$ is defined on a subset $E$ of a Banach space $Y$ we say
$f$ is locally Lipschitz on its domain $E$ if
for every $x\in E$ there exist $r>0$ and $L\geq0$ such that
$|f(y')-f(y)|\le L\|y-y'\|$ for all $y,y'\in E\cap\ball{x}{r}$; the smallest
such constant $L$ is called the Lipschitz constant of $f$ in $\ball{x}{r}$ and is
denoted $\text{Lip}(f|_{\ball{x}{r}})$.
A function $f:Y\to\real$ is simply called Lipschitz if there
is a common Lipschitz constant $L<\infty$ for which the Lipschitz condition
is satisfied for any pair of points $y,y'\in Y$.
The smallest
such constant $L \geq 0$ is then called the Lipschitz constant of $f$  and is
denoted by $\text{Lip}(f)$.

For any $f:Y\to\real$ and $y,e \in Y$,
we define the directional derivative of $f$ in the direction $e$ as
\begin{equation}\label{eq.dirder}
f'(y,e)=
\lim_{t\to0}\frac{f(y+te)-f(y)}{t}
\end{equation}
if the limit exists.
If, for a fixed $y\in Y$, the formula \eqref{eq.dirder}
 defines an element of $Y^*$, we
say $f$ is \gateaux{} differentiable at $y$.
Finally, if $f$ is \gateaux{} differentiable at $y$ and the convergence in
\eqref{eq.dirder} is uniform for $e$ in the unit sphere $S(Y)$ of $Y$,
we say that
$f$ is \frechet{} differentiable at $y$ and call
$f'(y)$ the \frechet{} derivative of $f$, where $f'(y)e = f'(y,e)$ for all $e \in Y$.

The main focus of the present paper is on
universal differentiability sets (UDS), those
subsets of a Banach space $Y$ that contain points of
\frechet{} differentiability of every Lipschitz function $f:Y\to\real$.

Recall a subset $P$ of $Y$ is called porous if there is a $c>0$ such that
for every $y\in P$ and every $r>0$ there exists $\rho<r$ and
$y'\in \ball{y}{\rho}$ such that $\ball{y'}{c\rho}\cap P=\emptyset$.
It is easy to see that any porous set is not a UDS since the distance function
$f(x)=\inf_{y\in P}{\|x-y\|}$ is $1$-Lipschitz and is not \frechet{}
differentiable at any point of $P$, provided $P$ is porous; \cite{Zdist}.
It turns out that the same is true for any $\sigma$-porous set $P$, that is any set that is a
countable union of porous sets; see \cite{BL}.

The existence of a non $\sigma$-porous set in Euclidean spaces without porosity points and with a null closure was first shown in \cite{Z76}; see also \cite{Z98,ZP}.  The set we are constructing will, in the finite dimensional  case, automatically be an example of such a set.

We shall be interested in the Hausdorff dimension of the universal differentiability sets we shall construct.
Recall, for $s\geq 0$ and $A\subseteq Y$
$$
\mathcal H^s (A)=
\lim_{\delta\downarrow0}\inf\Bigl\{\sum_i \textrm{diam}(E_i)^s \text{ where }
A\subseteq \bigcup_i E_i \text{, } \textrm{diam}(E_i)\le\delta\Bigr\},
$$
defines the $s$-dimensional Hausdorff measure of $A$, and
$$
\textrm{dim}_{\mathcal H}(A)=
\inf\{s \geq 0 \text{ such that } \mathcal H^s(A)=0\}
$$
the Hausdorff dimension of $A$.

Let $(M,\|\cdot\|)$ be a normed space and $t \in M^3$.
We call the union of segments
$$W(t) = [t_1,t_2]\cup[t_2,t_3]$$
a \textit{wedge} and define the standard \textit{wedge distance} by
$$d(W(t'),W(t))=\max_{1\le i\le 3}\|t_i'-t_i\|.$$
We call the space of all wedges equipped with the standard wedge distance
the \textit{wedge space} and denote it $(\mathcal W_M,d)$. Note
that the distance $d$ depends on the norm chosen on $M$.
For $\alpha>0$ and  subsets
$S_1,S_2\subseteq M$ we say $S_1$ is an
\textit{$\alpha$-wedge approximation} for
$S_2$ in norm $\|\cdot\|$
if for any
$W\in\mathcal W_M$ with $W\subseteq S_2$, there exists $W'\in\mathcal W_M$ with
$W'\subseteq S_1$ and $d(W',W)\le\alpha$.
When it is clear which norm on $M$ is considered we
shall just say that $S_1$ is an
\textit{$\alpha$-wedge approximation} for
$S_2$.
We shall also consider a more general construction when the
wedge space
is replaced by a general family of compact subsets of $M$, which may now be considered a general metric space. We shall at times make use of the Hausdorff distance between two such compact sets:
$$
\mathcal H(K_1,K_2)=\inf\{r>0:K_1\subseteq\ballcl{K_2}{r}
\quad\text{and}\quad
K_2\subseteq\ballcl{K_1}{r}\}.
$$
Here we use $\ballcl{A}{r}$ to denote the closed $r$-neighbourhood of $A \subseteq M$; we shall also use $\ball{A}{r}$ to denote an open $r$-neighbourhood of $A \subseteq M$.

In order to  construct a UDS
we first define a $G_\delta$ set $\G$ containing a dense
set of arbitrarily small wedges and then define a subset $S$ of $\G$
as described in Section~\ref{intro}.
For an arbitrary Lipschitz function we then apply our optimisation method
to $S$; see Section~\ref{Optimisation}. We remark that any $G_\delta$ set is a
complete topological space; this lets us conclude that the
differentiability point, which we find as a
limit point of the iterative construction, belongs to the set $S$.

As we have already mentioned in Section~\ref{intro}, any UDS has Hausdorff dimension
at least $1$. We prove this result in the next lemma.
\begin{lemma}\label{lem.hausd}
Let $Y$ be a non-zero Banach space and $S\subseteq Y$ a universal differentiability set.
Then the Hausdorff dimension of $S$ is at least $1$.
\end{lemma}
\begin{proof}
Assume $\dim_{\mathcal{H}}(S)<1$. Fix any nonzero $P\in Y^*$ and $e\in Y$ with
$P(e)=1$. The Hausdorff
dimension of $P(S)$ is strictly less than $1$, and therefore $P(S)$
has Lebesgue measure $0$.
Let  $g:\R\to\R$ be a Lipschitz function that is
not differentiable at any $y\in P(S)$.
Then
$f := g \circ P:Y\to\real$ is a Lipschitz function
that is not differentiable at any $x\in S$, as the directional derivative $f'(y,e)$ does not exist for $y \in S$.
\end{proof}

\section{Main results}
We begin this section with the statement of a
criterion for universal differentiability.

\begin{theorem}\label{th.main-part1}
Let
$(M,d)$ be a non-empty complete metric space,
$(Y,\|\cdot\|)$ be a Banach space with separable dual
and $\pi:M\to Y$ a continuous mapping.

Suppose that for every $\eta> 0$ and $x \in M$ and every open neighbourhood $N(x)$ of $x$ in $M$ there exists $\delta_0=\delta_0(x,N(x),\eta) > 0$ such that, for any $\delta \in(0, \delta_0)$ the set $\pi(N(x))$ is a $\delta\eta$-wedge approximation
for $\ball{\pi(x)}{\delta}$.

Then $\pi(M)$ is a universal differentiability set and, moreover,
for every Lipschitz function $g:Y\to\R$
the set
$D_g = \{y\in\pi(M):g\text{ is \frechet{} differentiable at }y\}$
is dense in $\pi(M)$. Furthermore, if $y \in \pi(M)$, $r>0$
and $P \colon Y \to \R$ is a non-zero continuous linear map then there exists a finite open interval $I = I_g(y)$ with $Py \in I$ and
$$
\mu(I \setminus P(D_g\cap\ball{y}{r})) = 0,
$$
where $\mu$ denotes the Lebesgue measure.
\end{theorem}

To prove Theorem~\ref{th.main-part1}, we need to find points
of \frechet{} differentiability in $\pi(M)$ for every Lipschitz function
defined on $Y$. To accomplish this, we first apply the
next theorem, Theorem~\ref{th.incr}, to obtain a point with
almost locally maximal directional derivative,
and then use
Differentiability Lemma~\ref{lem.diff} to  show that the function
is in fact \frechet{} differentiable at this point.

\begin{theorem}
\label{th.incr}
Let $(M,d)$ be a non-empty complete metric space, $(Y,\|\cdot\|)$ a Banach space,
$\pi\colon M \to Y$ a continuous map and
$\Theta\colon (0,\infty) \rightarrow (0,\infty)$ a real-valued function with
 $\Theta(t) \rightarrow 0$ as $t \rightarrow 0^+$.
 Assume $g\colon Y \to \R$ is a Lipschitz function and
$$
(x_0,e_0) \in
D = \{(x,e)\in M\times (Y\setminus \{0\}) \text{ such that  }g'(\pi x,e) \text{ exists}\}
$$
is such that $\|e_0\|=1$ and $g'(x_0,e_0)\ge0$.

Then one can define
\begin{enumerate}
\item[(1)]
a Lipschitz function $f \colon Y \to \R$ by
\begin{equation}\label{eq.deffg}
f=g+2\mathrm{Lip}(g)e_0^*,
\end{equation}
where $e_0^*\in Y^*$ is a linear functional such that
$\|e_0^*\|_{(Y,\|\cdot\|)^*}=e_0^*(e_0)=1$,
\item[(2)]
a norm $\| \cdot \|'$ on $Y$, with
$\|y\| \leq \| y \|' \leq 2 \|y\|$
for all $y \in Y$, and
\item[(3)]
a pair
$(\tilde x,\tilde e) \in D$ with $\|\tilde e \|' = 1$
\end{enumerate}
such that $f'(\pi\tilde x,\tilde e) \geq f'(\pi x_0,e_0)$ and the directional derivative $f'(\pi\tilde x,\tilde e)$ is almost locally maximal in the following sense.
For any
$\e > 0$ there exists an open neighbourhood $N_{\e}$ of $\tilde x$ in $M$ such that
whenever $(x',e') \in D$ with

\begin{enumerate}
\item[(i)]
$x' \in N_{\e}$, $\| e' \|' = 1$ and
\item[(ii)]
for any $t\in\R$
\begin{align}
\label{eqincr}
\notag |(f(\pi x'&+t\tilde e)-f(\pi x'))-(f(\pi\tilde x+t\tilde e)-
f(\pi\tilde x))|\\
&\leq \Theta(f'(\pi x',e')-f'(\pi\tilde x,\tilde e))|t|,
\end{align}
\end{enumerate}
then we have $f'(\pi x',e') < f'(\pi\tilde x,\tilde e) + \e$.

Moreover, if the original norm $\| \cdot \|$ is \frechet{} differentiable on $Y\setminus \{0\}$ then the norm $\| \cdot \|'$ can be chosen with this property too.
\end{theorem}

We prove Theorem~\ref{th.incr} at the end of Section~\ref{Optimisation}.
We will now use its conclusion to prove Theorem~\ref{th.main-part1}.

\begin{proof}[Proof of Theorem~\ref{th.main-part1}]
Without loss of generality we may assume that the norm $\|\cdot\|$
is \frechet{} differentiable on $Y\setminus\{0\}$, by~\cite{Asp, BL}, since passing to an equivalent norm keeps the
$\delta\eta$-wedge approximation condition and does not change the
differentiability property.

Taking arbitrary $x\in M$ and $N_0(x)=M$ we get that
the wedge approximation property of $\pi(N_0(x))$ implies that $\pi(M)$
contains a non-degenerate straight line segment $L\subseteq Y$. As
any Lipschitz function $g \colon Y \to \R$ is differentiable
at some point $p\in L$ in the direction of $L$, the set
$$
D:= \{(x,e) \in M \times (Y \setminus \{0\}) \text{ such that } g'(\pi x,e) \text{ exists}\}$$
is non-empty.

Without loss of generality we may assume that the Lipschitz constant of $g$
is equal to $1$. Picking an arbitrary $(x_0,e_0) \in D$ and
$\Theta(s) = 25 \sqrt{3s}$, we see that all the conditions
of  Theorem~\ref{th.incr} are satisfied if we rescale $e_0$ in order to
have $\|e_0\|=1$ and replace $e_0$ with $-e_0$ if necessary so as to have
$g'(x_0,e_0)\ge0$.
Let the Lipschitz function $f \colon Y \to \R$,
the norm $\|\cdot \|'$ on $Y$,
the pair $(\tilde x,\tilde e) \in D$
and, for each $\e > 0$, the open neighbourhood $N_{\e}$ of $\tilde x \in M$ be given by the conclusion of Theorem~\ref{th.incr}. Note that
$f'(\pi\tilde x,\tilde e)\ge f'(\pi x_0,e_0)$,
$\text{Lip}(f) \leq 3$, we may take $\|\cdot \|'$ to be \frechet{} differentiable on $Y\setminus \{0\}$ and that
 \begin{equation}
 \label{normbound}
 \|z\| \leq \|z\|' \leq 2\|z\|
 \end{equation}
 for all $z \in Y$, so that $\|\tilde e\| \leq \|\tilde e\|' = 1$.

We claim that $\tilde y = \pi\tilde x$ is a point of \frechet{}
differentiability of $f$.

Since the two norms $\|\cdot\|$, $\|\cdot\|'$ are equivalent, it suffices
to verify the conditions of
Lemma~\ref{lem.diff} for $(Y,\|\cdot\|')$, applied to the
Lipschitz function $f$, $L = 3$ and the pair
$$(\tilde y,\tilde e) = (\pi \tilde x,\tilde e).$$
To accomplish this, we let $\e,\theta>0$ and show that
$$F_\e=\pi(N_\e) \text{ and } \delta_*=\delta_0(\tilde x,N_\e,\theta/2)$$
are such that (1) and (2) of Lemma~\ref{lem.diff} hold, with the norm $\|\cdot\|$ replaced by $\|\cdot\|'$.

Suppose $\delta \in (0,\delta_*)$, $\|y_i-\tilde y\|'<\delta$ for $i=1,2,3$.
Then
from~\eqref{normbound} we have $\|y_i-\tilde y\| < \delta$
for $i=1,2,3$ as well. Now using that $F_\e=\pi(N_\e)$ is
a $\delta\theta/2$-wedge approximation for $\ball{\tilde y}{\delta}$ in $\|\cdot\|$
and the inequality \eqref{normbound}, we get that $F_\e$
is a $\delta\theta$-wedge approximation for $\ball{\tilde y}{\delta}$ in
norm $\|\cdot\|'$.
This verifies condition (1) of Lemma~\ref{lem.diff}.

For condition (2) we note that if $y' \in F_\e$, $\|e'\|' = 1$ and
\begin{align*}
&|(f(y'+t\tilde e)-f(y'))-(f(\tilde y+t\tilde e)-f(\tilde y))|\\
&\leq 25\sqrt{(f'(y',e')-f'(\tilde y,\tilde e))L}\cdot |t|
\end{align*}
for all $t \in \R$, then as $F_\e = \pi(N_\e)$ we may write
$y' = \pi x'$ where $x' \in N_\e$. As $L = 3$ and $\Omega(s) = 25 \sqrt{3s}$, the conditions (i) and (ii) of Theorem~\ref{th.incr} are satisfied, so we deduce that
$f'(\pi x',e') < f'(\pi\tilde x,\tilde e)+\e$.

As all the conditions of Lemma~\ref{lem.diff} are satisfied we deduce that $f$ is \frechet{} differentiable at $\tilde y = \pi\tilde x \in \pi(M)$. As $f- g$ is
linear, we conclude
$g$ is also \frechet{} differentiable at $\tilde y \in \pi(M)$. Hence $\pi(M)$ is indeed a universal differentiability set in $Y$.

Note moreover we have proved slightly more: namely, if $M$ is any non-empty complete metric space
satisfying the wedge approximation property as in the conditions of present
theorem, then for any Lipschitz $g:Y\to\real$ and
an arbitrary pair $(x_0,e_0)\in M\times (Y\setminus\{0\})$
such that $\|e_0\|=1$ and $g'(\pi x_0,e_0)\ge0$, there is a Lipschitz function
$f:Y\to\real$ defined according to \eqref{eq.deffg}
and a pair $(\tilde x,\tilde e)\in M\times (Y\setminus\{0\})$ such that
$\|\tilde e\|\le1$,
$g$ is \frechet{}
differentiable at $\pi\tilde x$ and
$f'(\pi\tilde x,\tilde e)\ge f'(\pi x_0,e_0)$.

To verify the
density of the set
$D_g = \{y\in\pi(M):g\text{ is \frechet{} differentiable at }y\}$
in $\pi(M)$,
it suffices to note that if $y=\pi x \in \pi(M)$ and $\e > 0$, we may pick a non-empty open set $N \subseteq M$ such that
$\pi(N) \subseteq \ball{y}{\e}$.
Then as the restriction bundle $\pi|_{N} \colon N \to Y$ satisfies the conditions of the present theorem, any Lipschitz $g \colon Y \to \R$ contains a point of \frechet{} differentiability in
$\pi(N) \subseteq \pi(X) \cap \ball{y}{\e}$.

We now check the last observation of the theorem. We may assume $\|P\| = 1$.
Let
$$y =\pi(x)\in \pi(M)\text{ and }\delta_0=\delta_0(x,M,\eta),$$
where $\eta\in (0,1/2)$.
Choose also a
vector $e_1\in Y$ such that $Pe_1=1$.
Fix any
$$\delta\in(0,\min\{1/30,r/2,\delta_0\})$$
and find a line
segment $L_0\subseteq \pi(M)$ that is
an $\eta\delta$-wedge
approximation for $L=[y-\delta e_1,y+\delta e_1]$.
It is easy to see that
$L_0\subseteq\ball{y}{r}$ and
$$
P(L_0)\supseteq I=(Py-(1-\eta)\delta,Py+(1-\eta)\delta).
$$
Let $e_0$ be the unit vector in the direction of $L_0$.
As $g$ is Lipschitz, the directional derivative
$g'(z,e_0)$ exists for almost all points $z\in L_0$.
We note that the set $D_g$ is a $F_{\sigma\delta}$-set:
$$
D_g=
\bigcap_{n\ge1}
\bigcup_{\underset{\delta\in\mathbb Q}{y^*\in A}}
\bigcap_{\underset{|t|<\delta}{\|z\|\le1}}
\bigl\{y\in Y\colon
|g(y+tz)-g(y)-ty^*(z)|\le|t|/n
\bigr\},
$$
where $A$ is a countable dense subset of the unit ball of $Y^*$.
Therefore
the image
$$P(D_g \cap \ball{y}{r}),$$
being a projection of a Borel subset of
a Polish space,
is an analytic subset of $\R$ and therefore Lebesgue measurable.

Suppose then that the Lebesgue measure $\mu(I \setminus P(D_g\cap\ball{y}{r}))$
is strictly positive.
There exists a non-constant everywhere differentiable Lipschitz function
$h \colon I \to \R$ such that $h' = 0$ on $P(D_g\cap\ball{y}{r}) \cap I$.
This implies there exists $y_0\in L_0$ such that $s=Py_0\in I$,
the directional derivative $g'(y_0,e_0)$ exists and $h'(s)\ne0$.
By scaling $h$ if necessary we may assume $h'(s) = 1$.
Let $G = g + 3h\circ P$. This is a Lipschitz mapping defined on a
$G_\delta$-set $\widetilde M=\pi^{-1}(L_0\cap P^{-1}(I))\subseteq M$ and such that the directional
derivative
$G'(y_0,e_0)$ exists; moreover,
$G'(y_0,e_0)=g'(y_0,e_0)+3\ge2$ as $\text{Lip}(g) = 1$.
Since $L_0\subseteq \pi(\widetilde M)$ we can find $x_0\in\widetilde M$
such that $y_0=\pi x_0$.

Then, using the more general statement we have proved for the first part
of the theorem for $\widetilde M$ instead of $M$ and $G$ instead of $g$
we conclude that there is a
Lipschitz function
$F:Y\to\real$ defined according to \eqref{eq.deffg},
$F=G+2\mathrm{Lip}(G)e_0^*$, where $\|e_0^*\|=e_0^*(e_0)=1$,
and a pair
$(\tilde x,\tilde e)\in\widetilde M\times (Y\setminus\{0\})$ such that
$\|\tilde e\|\le1$,
$G$ is \frechet{}
differentiable at $\pi\tilde x$ and
$F'(\pi\tilde x,\tilde e)\ge F'(\pi x_0,e_0)$.
Then $\tilde y=\pi\tilde x\in L_0\subseteq\ball{y}{r}$
is a point of \frechet{} differentiability of $G$ and
\begin{align*}
G'(\tilde y,\tilde e)-G'(y_0, e_0)
&=
F'(\tilde y,\tilde e)-F'(y_0, e_0)
+2\mathrm{Lip}(G) e_0^*(e_0-\tilde e)\\
&\ge
F'(\tilde y,\tilde e)-F'(y_0, e_0)
\ge0
\end{align*}
as $e_0^*(e_0)=1$ and $e_0^*(\tilde e)\le\|\tilde e\|\le1$.
Together with $G'(y_0,e_0)\ge2$ we conclude
$G'(\tilde y,\tilde e)\ge2$.

However, since $G$ is \frechet{} differentiable at $\tilde y$, so is $g$, and
therefore $\tilde y\in D_g\cap\ball{y}{r}$.
As we also have $\tilde y=\pi\tilde x\in\pi(\widetilde M)$, we conclude
$P\tilde y\in P(D_g\cap\ball{y}{r})\cap I$; hence
$G'(\tilde y,\tilde e)=g'(\tilde y,\tilde e)$,
a contradiction to $G'(\tilde y,\tilde e)\ge2$
as a directional derivative of a $1$-Lipschitz function $g$
cannot exceed $1$.
\end{proof}

Together with the following statement, Theorem~\ref{th.main-part1}
implies the existence of a closed universal differentiability set; see
Lemma~\ref{lem.main-cl}.
\begin{theorem}
\label{th.passtoclosed}
Let $(Y,d)$ be a metric space and let $(K_r)_{r \in R}$ be a collection of non-empty compact subsets of $Y$ indexed by a non-empty metric space $(R,\gamma)$ such that the Hausdorff distance $\mathcal{H}(K_r,K_s)$
is bounded from above by
$\gamma(r,s)$ for every $r,s \in R$.
Assume $\G$ is a $G_{\delta}$ subset of $Y$ such that $\G$ contains a $\gamma$-dense subset of the family $(K_r)_{r \in R}$ and $K_{r_0}\subseteq \G$ is one
of these compacts.
Assume further that there exist $\rho>0$ and $\e_0 > 0$ such that for every $\e \in (0,\e_0)$ we can find a set of indices
$R(\e) \subseteq R$ such that
\begin{equation}\label{eq.repsilon}
\begin{tabular}{l}
\textbullet\quad for every $s \in R$ there exists $t \in R(\e)$ with $\gamma(t,s) < \e$,\\
\textbullet \quad
for every subset $S$ of $Y$ of diameter at most $\rho \e$
the set\\ \quad\ $\{r \in R(\e): S\cap K_r\ne\emptyset\}$ is finite.
\end{tabular}
\end{equation}
Then there exists a nested collection of closed non-empty subsets $(\T_{\lambda})_{0 \leq \lambda \leq 1}$ of $\G$,
$$\T_{\lambda'} \subseteq \T_{\lambda} \text{ whenever } 0\le \lambda'\leq \lambda\le1,$$
each containing $K_{r_0}$
that satisfies the following. For each
$$\eta > 0\text{, }\lambda\in (0,1]\text{ and }y \in \bigcup_{0\le\lambda'<\lambda}\T_{\lambda'}$$
there exists
$\delta_1= \delta_1(\eta,\lambda,y) > 0$ such that if
$\delta \in (0,\delta_1)$ and $s \in R$ with $K_s\subseteq\ballcl{y}{\delta}$ there exists $t \in R$ such that $K_t \subseteq \T_{\lambda}$ and
$\gamma(t,s) < \eta \delta$.
\end{theorem}

\begin{remark}\label{rem.wedge}
We prove Theorem~\ref{th.passtoclosed} in Section~\ref{Set_theory}.

Let now
$R = Y \times Y \times Y$ and for each $r = (y_1,y_2,y_3) \in R$
define
$$K_r = W(r)=[y_1,y_2] \cup [y_2,y_3].$$
If we further let $\gamma(K_r,K_s)$ be equal to the standard wedge
distance,
$$\gamma(K_r,K_s)=d(W(r),W(s)),$$
the conclusion of
Theorem~\ref{th.passtoclosed}
is: there exists $\delta_1>0$ such that if $\delta\in(0,\delta_1)$
then $\T_\lambda$ is a $\eta\delta$-wedge approximation for
$\ballcl{y}{\delta}$.
In Lemma~\ref{lem.main-cl}
we show that
this property implies that
$\T_\lambda$ are universal differentiability sets.
We will later easily get that $T_\lambda$ has Hausdorff dimension $1$
by taking the containing $G_\delta$-set $\G$ of Hausdorff dimension $1$.
See Lemma~\ref{lem.gdelta} for the list of properties that we require
$\G$ to satisfy for this.

However, in order to get the conclusion of Theorem~\ref{th.passtoclosed},
one needs to verify condition \eqref{eq.repsilon}.
In the case in which $Y$ is a finite dimensional space, it is easy to see that
since balls in $Y$ are totally bounded sets,
$R=Y^3$
satisfies the required condition.
In case $Y$ is an infinite dimensional space, we
prove this property in Lemma~\ref{lem.reduction}.
\end{remark}

\begin{lemma}\label{lem.main-cl}
Let $Y$ be a Banach space
with separable dual and $(\mathcal W,d)=(\mathcal W_Y,d)$ be the wedge space equipped with the standard wedge distance.
Suppose $\G$ is a $G_\delta$ subset of $Y$ containing a $d$-dense subset of
$\mathcal W$ and the nested collection $(T_\lambda)_{0\le\lambda\le1}$ of
non-empty closed subsets of $Y$, $T_{\lambda'}\subseteq T_{\lambda}$ for
$0\le\lambda'\le\lambda\le1$, satisfies the condition that for
each $\eta>0$, $\lambda\in(0,1]$ and
$y\in\bigcup_{0\le \lambda'<\lambda} T_{\lambda'}$
there is a
$\delta_1=\delta_1(\eta,\lambda,y)>0$
such that for all $\delta\in(0,\delta_1)$ the set $T_\lambda$ is a
$\eta\delta$-wedge approximation for $\ball{y}{\delta}$.

Then
for each $\lambda\in(0,1]$
the set $\T_\lambda$  is a closed universal differentiability set.
Furthermore, for any Lipschitz function $g:Y\to\R$, any
$x\in T_{\lambda'}$, $0\le \lambda' < \lambda\le1$, $r>0$ and any
non-zero continuous linear map
$P \colon Y \to \R$ there exists a finite open interval $I = I_g(x)$ with
$Px \in I$ and
$$
\mu(I\setminus P(T_{\lambda}\cap D_{g,r}(x))) = 0,
$$
where $D_{g,r}(x)$ is the set of points of \frechet{}
differentiability of $g$ in the $r$-neigh\-bour\-hood of $x$ and
$\mu$ denotes the Lebesgue measure.
\end{lemma}
\begin{proof}
For every $\lambda\in(0,1]$, define a subset of $(0,\lambda)\times Y$
\begin{equation}\label{eq.defxl}
X_\lambda=\{(\tau,y):
0<\tau<\lambda\text{ and }
y\in \T_{\tau'}\text{ for every }\tau'\in(\tau,1)\}.
\end{equation}
Note that if
 $\tau\in(0,\lambda)$ we have
$X_\lambda\supseteq\{\tau\}\times\T_{\tau}$, so $X_\lambda\ne\emptyset$;
and
for every $(\tau,y)\in X_\lambda$ we necessarily have $y\in\T_\lambda$.
Moreover, if we let $\Delta$ denote a complete metric on
$(0,\lambda)$, then
$$
d((\tau',y'),(\tau,y))=\Delta(\tau',\tau)+\|y'-y\|
$$
makes $X_\lambda$ a  complete metric space, since $T_\lambda$ is closed.

We now check that the conditions of Theorem~\ref{th.main-part1}
are satisfied for
$$M=X_\lambda \text{ and }\pi(\tau,y)=y.$$
Assume we are given $\eta>0$, a point $x=(\tau,y)\in X_\lambda$
and its open neighbourhood $N(x)$. Without loss of generality we may assume there is $\psi>0$ such that
$$
N(x)=\{(\tau',y')\in X_\lambda:\Delta(\tau',\tau)<\psi
\text{ and }\|y'-y\|<\psi\}.
$$
Then fixing $\tau'\in(\lambda,\tau)$ such that
$\Delta(\tau',\tau)<\psi$ we get
$\pi(N(x))\supseteq \ball{y}{\psi}\cap\T_{\tau'}$.
Define now
$\delta_0(x,N(x),\eta)=\min\{\delta_1(y,\tau',\eta),\psi/2\}$
and assume $\delta\in(0,\delta_0)$. Since $T_{\tau'}$
is a $\delta\eta$-wedge approximation for  $\ball{y}{\delta}$ and
$\delta+\delta\eta<2\delta<\psi$, we conclude that
$\T_{\tau'}$, and therefore
$\pi(N(x))$ as well is a
$\delta\eta$-wedge approximation for  $\overline{B}_\delta(x)$.

The conclusion of Theorem~\ref{th.main-part1} says that
$\pi(X_\lambda)$ is a universal differentiability set. Since
$\pi(X_\lambda)\subseteq\T_\lambda$ we conclude
$\T_\lambda$ is a universal differentiability set,
for every $\lambda\in(0,1]$.

Moreover, if $x\in T_{\lambda'}$ and $0\le\lambda'<\lambda\le1$
we conclude $(\lambda',x)\in X_{\lambda}$
(if $\lambda'=0$ then find $\lambda''\in(0,\lambda)$ and get
$x\in\T_{\lambda''}$ so
$(\lambda'',x)\in X_{\lambda}$).
Then the final part of the lemma follows from the conclusion of
Theorem~\ref{th.main-part1}.
\end{proof}

Lemma~\ref{lem.reduction} shows that most natural choices of $(R,\gamma)$ in $Y$, an infinite dimensional separable Banach space, satisfy the conditions of Theorem~\ref{th.passtoclosed} with $\rho = 1/4$; in particular the conditions are satisfied whenever the collection $(K_r)_{r \in R}$ of compacts is translation invariant, with $\gamma(K_r,x+K_r)\le\|x\|$.

\begin{lemma}
\label{lem.reduction}
Suppose $(Y,\|\cdot\|)$ is an infinite dimensional Banach space, $(R,\gamma)$ is separable and has the property that whenever $r \in R$ and $x \in Y$ then $K_s = x+K_r$ for some $s \in R$ with
\begin{equation}
\label{yacond}
\gamma(s,r) \leq \frac{1}{4\rho}\|x\|.
\end{equation}
Then for every $\e > 0$ there exists a set $R(\e) \subseteq R$ such that
\begin{enumerate}
\item
\label{iunnoone}
for all $r \in R$ there exists $s \in R(\e)$ with $\gamma(s,r) < \e$,
\item
\label{iunnotwo}
if $r,s$ are distinct elements of $R(\e)$ then $\text{dist}(K_r,K_s) > \rho\e$,
\end{enumerate}
where for compact $K,K' \subseteq Y$, we define
$$\text{dist}(K,K') = \inf\{\|k'-k\| \text{ where }k \in K\text{, }k' \in K'\}.$$
\end{lemma}

We establish the lemma in a few short steps.

\begin{lemma}
\label{porosityplus}
If $Y$ is an infinite dimensional Banach space and $K \subseteq Y$ is compact then for every $\e > 0$ there exists $y \in Y$ with $\|y\| = \e$ and $\text{dist}(y,K) > \e/3$.
\end{lemma}
\begin{proof}
It is well known that one may find an infinite collection $(e_n)_{n \in \mathbb{N}}$ in $Y$ with $\|e_n\| = 1$ and $\|e_n-e_m\| \geq 1$ for $m \neq n$. Assuming, for a contradiction, that we cannot find $n$ with
$\text{dist}(\e e_n,K) > \e/3$ then we can pick $k_n \in K_n$ for each $n$ with $\|k_n-\e e_n\| \leq \e/3$. It then follows that $\|k_n-k_m\| \geq \e/3$ for all $m \neq n$, contradicting the compactness of $K$.
\end{proof}

\begin{lemma}
\label{perturbcompacts}
If $Y$ is an infinite dimensional Banach space and $(K_n)_{n \geq 1}$ are compact subsets of $Y$ then for any $\e > 0$ we can find $y_n \in Y$ with $\|y_n\| = \e$ for each $n \geq 1$ such that the translates $K_n' := y_n + K_n$ satisfy $\text{dist}(K'_n,K'_m) > \e/3$ for $n \neq m$.
\end{lemma}

\begin{proof}

Suppose $n \geq 1$ and we have chosen $(y_m)_{1 \leq m < n}$ such that $\text{dist}(K'_m,K'_l) > \e/3$ for $1 \leq l < m < n$. It suffices to pick $y_n$ such that $\text{dist}(K'_n,K'_m) > \e/3$ for $1 \leq m < n$.

The difference set
$$K := K_{n}-\cup_{1\leq m < n} K'_m = \{k-k' \text{ where }k \in K_{n},\text{ } k' \in K_m \text{ for some } m < n\}$$
is compact so that we may find $y \in Y$ with $\|y\| = \e$ and $\text{dist}(y,K) > \e/3$, using Lemma~\ref{porosityplus}. Then $\text{dist}(0,-y+K) > \e/3$ so that, choosing $y_n = -y$,
\begin{equation*}
\text{dist}(0,K'_n-\cup_{1\leq m < n} K'_m) > \e/3.
\qedhere
\end{equation*}
\end{proof}
\noindent
\textit{Proof of Lemma~\ref{lem.reduction} }
We may assume $R \neq \emptyset$. Let $(r_n)_{n \geq 1}$ be a dense sequence in $R$. By Lemma~\ref{perturbcompacts} we can find $y_n \in Y$ with $\|y_n\| = 3\rho \e$ such that $K'_n := y_n + K_{r_n}$ satisfy $\text{dist}(K'_n,K'_m) > \rho \e$ for $n \neq m$. Now we may pick $r_n'$ with $K_{r_n'} = y_n + K_{r_n} = K_n'$ and
$$\gamma(r_n',r_n) \leq \frac{1}{4\rho}\|y_n\| = \frac{3}{4}\e$$
using~\eqref{yacond}. Setting $R(\e) = \{r_n' \text{ where } n \in \mathbb{N}\}$
we are done. \qed

\vspace{2ex}
\noindent
\textbf{Conclusion.}
We summarise what we have shown and add some further observations. First note, Lemma~\ref{porosityplus} implies that any compact set in an infinite dimensional
space is porous. Now, as any porous
set is not a UDS,
it follows that a UDS cannot be compact in infinite dimensional
spaces.

On the other hand, we now show that
inside any non-empty open set in $Y$ we can find
a closed universal
differentiability set of Hausdorff dimension $1$
which
does not contain any continuous curves: this set can
be chosen to be totally disconnected.

\begin{lemma}\label{lem.gdelta}
Let $Y$ be a non-zero separable Banach space
 and $(\mathcal W,d)=(\mathcal W_Y,d)$ be the wedge space equipped with the standard wedge distance.
Then given any $\phi\in Y^*\setminus\{0\}$
there exists
a  $G_\delta$ subset $\G$ of $Y$ of Hausdorff dimension $1$ such that
$\G$ contains a $d$-dense subset of
$\mathcal W$ and the intersection
$\G\cap(y+\ker\phi)$ is totally disconnected for any $y\in Y$.
\end{lemma}
\begin{proof}
Let $\mathcal W_0\subseteq\mathcal W$ be a $d$-dense countable subset.
Note that
$$
\mathcal W_1=\{W=(y_1,y_2,y_3)\in\mathcal W_0:
\phi(y_1)\ne\phi(y_2)\text{ and }\phi(y_2)\ne\phi(y_3)\}
$$
is then also
$d$-dense in $\mathcal W$.
Let $\G'\subseteq Y$ be a $G_\delta$ set  of Hausdorff dimension $1$
such that $W\subseteq \G'$ for all $W\in\mathcal W_1$.
Let further $\{W_1,W_2,\dots\}$ be an enumeration of $\mathcal W_1$.

Let $L=y_0+\real e$ be a line through one of the sides of a wedge
$W\in\mathcal W_1$ and $a>0$. Then, for any $y\in Y$,
the diameter of the intersection of
$\ball{L}{a}$ with $y+\ker\phi$ does not exceed $2a(1+\|\phi\|/|\phi(e)|)$. Therefore
if we let the countable set $(e_{1,i},e_{2,i})_{i\ge1}$ be the
pairs of directions of sides of all wedges in $\mathcal W_1$, the open set
$\G_n=\bigcup_{i\ge1} \ball{W_i}{\e_i}$ intersects $y+\ker\phi$
for any $y\in Y$ by a set of diameter less than $1/n$, whenever
$$\displaystyle 0<\e_i<1/\Bigl(n2^{i+1}\bigl(1+\frac{\|\phi\|}
{\min\{|\phi(e_{1,i})|,|\phi(e_{2,i})|\}}\bigr)\Bigr).$$

Thus the conclusion of the lemma is satisfied for $\G=\G'\cap\bigcap_{n\ge1}\G_n$.
\end{proof}
\begin{theorem}
\label{th.UDtot}
Let $Y$ be a  Banach space with separable dual. Then for every open
set $U\subseteq Y$ and $y_0\in U$ there is a closed
set $S\subseteq U$ of Hausdorff
dimension $1$ such that $y_0 \in S$ and
every locally Lipschitz function
$f$ defined on a domain containing $U$ has a point of
\frechet{} differentiability inside $S$. Moreover, the set
$S$ may be chosen to be in addition
totally disconnected so that it contains no non-constant continuous curves.
\end{theorem}
\begin{proof}
Fix any non-zero continuous linear map $P:Y\to\real$.
Let $\G$ be a $G_\delta$ subset of $Y$ satisfying Lemma~\ref{lem.gdelta}
with $\phi=P$.
By Remark~\ref{rem.wedge} and Lemma~\ref{lem.reduction}
we can apply Theorem~\ref{th.passtoclosed} in order to get a nested sequence
of closed sets $T_\lambda\subseteq \G$ satisfying the hypothesis of
Lemma~\ref{lem.main-cl}.

By translation we may assume without loss
of generality $y_0\in T_\lambda$
for some $\lambda\in[0,1)$. Let
$\lambda_0\in(\lambda,1]$
and  $r_0>0$ be such that
$\ball{y_0}{r_0}\subseteq U$; define $S_1=\T_{\lambda_0}\cap\ballcl{y_0}{r_0/2}$.

Let $C\subseteq[0,1]$ be a closed totally disconnected set of
positive measure, such that every neighbourhood of any of its points intersects
$C$ by a set of positive measure. An example of such set could be
a Cantor set of positive measure.

Let $C_0$ be a shift of $C$ such that $Py_0\in C_0$.
Consider a set
$$
S=P^{-1}(C_0)\cap S_1
=P^{-1}(C_0)\cap\T_{\lambda_0}\cap\ballcl{y_0}{r_0/2}.
$$
We clearly have $y_0\in S\subseteq \ballcl{y_0}{r_0/2}\subseteq U$.
Note further that as $P^{-1}(c)\cap \G$ is totally disconnected
for every $c\in C$, and $C$ is totally disconnected by itself, the set
$S$ set is totally disconnected. It is also clear $S$ is closed and
$\text{dim}_{\mathcal H}(S)\le1$ as $\text{dim}_{\mathcal H}(\G)=1$
and $S\subseteq \G$.
It remains to verify that
every locally Lipschitz function defined on a domain containing $U$
has a point of differentiability in $S$. By Lemma~\ref{lem.hausd}
this would also imply $\text{dim}_{\mathcal H}(S)=1$.

Let $f:U'\to\real$ be a locally Lipschitz function
with domain $U'$ containing $U$.
Then for the restriction $f|_{\ball{y_0}{r_0}}$ there exists a Lipschitz
extension $\tilde f$ to the whole space $Y$; one can take for example
$\tilde f(x)=\inf_{y\in\ball{y_0}{r_0}}(f(y)+L\|y-x\|)$, where
$L\ge\text{Lip}(f|_{\ball{y_0}{r_0}})$.

Let $D_{\tilde f}$ be the set of points of \frechet{} differentiability of
$\tilde f$ inside $T_{\lambda_0}$.
By Lemma~\ref{lem.main-cl} there exists a finite open interval
$I=I_{\tilde f}(y_0)\ni Py_0$ such that almost every point in $I$
belongs to $P(S_1\cap D_{\tilde f})$.
Since $Py_0\in C_0$,
we can find a nearby point that belongs to $C_0\cap P(S_1\cap D_{\tilde f})$.
This means $P^{-1}(C_0)$ intersects $S_1\cap D_{\tilde f}$. As
$f$ coincides with $\tilde f$ on $\ball{y_0}{r_0/2}$ and
$D_{\tilde f}\subseteq\T_{\lambda_0}$ we conclude
there is a point of \frechet{} differentiability of $f$
that belongs to
$$P^{-1}(C_0)\cap\T_{\lambda_0}\cap\ballcl{y_0}{r_0/2}=S.$$
\end{proof}

\section{Differentiability}
\label{Differentiability}
We start this section by quoting \cite[Lemma 4.2]{DM}:

\begin{lemma}\label{lem.max}
Let $(Y,\|\cdot\|)$ be a Banach space,
$f\colon  Y \rightarrow \mathbb{R}$ be a Lipschitz function with Lipschitz
constant $\mathrm{Lip}(f) > 0$ and let $\e\in(0,\mathrm{Lip}(f)/9)$.
Suppose $y\in Y$, $e\in S(Y)$ and $s > 0$ are such that
the directional derivative $f'(y,e)$ exists, is non-negative and
\begin{equation}
\label{eq.80lip}
|f(y+te)-f(y)-f'(y,e)t| \leq \frac{\e^{2}}{160\mathrm{Lip}(f)} |t|
\end{equation}
for $|t| \leq s\sqrt{\frac{2\mathrm{Lip}(f)}{\e}}$.
Suppose further $\xi\in(-s/2,s/2)$
and $\lambda \in Y$ satisfy
\begin{align}
\label{eq.est160}&|f(y+\lambda)-f(y+\xi e)| \geq 240\e s,\\
\label{eq.lambdaxi}&\|\lambda-\xi e\| \leq s\sqrt{\frac{\e}{\mathrm{Lip}(f)} }\\
\label{eq.estup}\text{and \quad}&\frac{\|\pi se+\lambda\|}{|\pi s+\xi|} \leq 1 + \frac{\e}{4\mathrm{Lip}(f)}
\end{align}
for $\pi = \pm 1$. Then if $s_{1},s_{2},\lambda' \in Y$ are such that
\begin{equation}\label{eq.ests1s2}
\max(\|s_{1}-se\|,\|s_{2}-se\|) \leq \frac{\e^{2}}{320\mathrm{Lip}(f)^{2}}s
\end{equation}
and
\begin{equation}\label{eq.estlam}
\|\lambda'-\lambda\| \leq \frac{\e s}{16\mathrm{Lip}(f)},
\end{equation}
we can find
$y' \in [y-s_1,y+\lambda'] \cup [y+\lambda',y+s_2]$ and $e' \in S(Y)$ such
that the directional derivative $f'(y',e')$ exists,
$f'(y',e') \geq f'(y,e) + \e$
and for all $t \in \mathbb{R}$ we have
\begin{align}\label{eq.condi}
&|(f(y'+te)-f(y'))-(f(y+te)-f(y))|\\
&\notag\leq
25\sqrt{(f'(y',e')-f'(y,e))\mathrm{Lip}(f)}|t|.
\end{align}
\end{lemma}

Our next lemma is crucial for the proof of Theorem~\ref{th.main-part1}
and enables us to demonstrate the universal differentiability property of
the set by finding a point $y$ with almost maximal directional derivative
and a family of sets around $y$ with wedge approximation for arbitrarily
small balls around $y$.
See the definition of wedge approximation in Section~\ref{sec.def}.

\begin{lemma}[Differentiability Lemma]\label{lem.diff}
Let $(Y,\|\cdot\|)$ be a Banach space such that the norm $\|\cdot\|$
is \frechet{} differentiable on $Y\setminus\{0\}$.
Let $f:Y\to\real$ be a Lipschitz function and
$$(y,e)\in Y\times S(Y)$$
be such that the directional derivative
$f'(y,e)$ exists and is non-negative. Suppose that there is a family of
sets $\{F_\e\subseteq Y:\e>0\}$ such that
\begin{enumerate}
\item
whenever $\e,\theta>0$ there exists $\delta_*=\delta_*(\e,\theta)>0$ such that for any
$\delta\in(0,\delta_*)$
the set $F_\e$ is a $\delta\theta$-wedge approximation for
$\ball{y}{\delta}$.
\item
whenever $(y',e')\in F_\e\times S(Y)$ is such that the directional derivative
$f(y',e')$ exists, $f'(y',e')\ge f'(y,e)$ and
for any $t\in\real$ \eqref{eq.condi} is satisfied, i.e.
$$
|(f(y'+te)-f(y'))-(f(y+te)-f(y))|
\le25\sqrt{(f'(y',e')-f'(y,e))\mathrm{Lip}(f)}|t|,
$$
then $f'(y',e')<f'(y,e)+\e$.
\end{enumerate}
Then $f$ is \frechet{} differentiable at $y$.
\end{lemma}
\begin{proof}
We may assume $\mathrm{Lip}(f) = 1$.
Note that the norm $\|\cdot\|$ is differentiable at $e$, let $e^*$
be its derivative at $e$.
We shall prove that $f$ is \frechet{} differentiable at $y$ and that
$f'(y)$ is given by the formula
$$
f'(y)(u)=f'(y,e)e^*(u).
$$

Note that $\|e^*\|=1$ and $e^*(e)=1$. Fix
arbitrary $\eta\in(0,1/3)$. Choose $\Delta\in(0,\eta)$ such that
\begin{equation}\label{eq.normdif}
\Bigl|
\|e+th\|-\|e\|-te^*(h)\Bigr|
\le\eta|t|
\end{equation}
for any $\|h\|\le1$ and $|t|\le\Delta$.

Let $\e=\eta\Delta$ and $\theta=\eta^2\Delta^2/320$.
We know that the directional derivative $f'(y,e)$ exists so that there
we may pick
$\rho \in(0, \delta_*(\e,\theta))$
such that whenever $|t| <\rho$,
\begin{equation}\label{eq.diffxe}
|f(y+te)-f(y)-f'(y,e)t| < \frac{\eta^2\Delta^2}{160}|t|.
\end{equation}

Let $\delta=\frac{1}{32}\rho\Delta\sqrt{\Delta\eta}$.
We plan to show that
\begin{equation}\label{difR}
|f(y+ru)-f(y)-f'(y,e)e^*(u) r| < 5000\eta r
\end{equation}
for any $\|u\|\le1$ and $r\in(0,\delta)$.
This will imply the differentiability of $f$ at $y$.

Assume for a contradiction, that there exist $r\in(0,\delta)$
and $\|u\|\le1$ such that the inequality
\eqref{difR} does not hold:
\begin{equation}\label{nondif}
|f(y+ru)-f(y)-f'(y,e) e^*(u) r|
\geq 5000\eta r.
\end{equation}

Define
\begin{equation*}
s=16 r/\Delta,\qquad
\lambda=ru \qquad
\text{ and }\qquad
\xi=r e^*(u).
\end{equation*}
We check now that all the conditions of
Lemma~\ref{lem.max} are satisfied with
$\e,s,\xi,\lambda$ defined as above.

First of all,
$\e=\eta\Delta<1/9$ and
condition \eqref{eq.80lip} follows from
\eqref{eq.diffxe} as ${\e^2}=\eta^2\Delta^2$ and
$s\sqrt{2/\e}=16\sqrt{2}r/(\Delta\sqrt{\eta \Delta})<
32\delta/(\Delta\sqrt{\eta \Delta})
=\rho$.

Next we check $|\xi|<s/2$ and condition \eqref{eq.est160}.
Indeed, $|\xi|\le r< r/\Delta=s/16<s/2$.
Moreover, $r\le\delta<\rho$,
so that we may apply \eqref{eq.diffxe} with $t=\xi$. Combining this
with \eqref{nondif} we verify condition \eqref{eq.est160}:
$$
|f(y+ru)-f(y+\xi e)|
\ge
5000\eta r-\eta r\frac{\eta\Delta^2}{160}|e^*(u)|
\ge
240\cdot16\eta r
=240s\Delta\eta
=240s\e.
$$

Note $\|\lambda-\xi e\|=r\|u-e^*(u)e\|\le 2r
<16r\sqrt{\eta/\Delta}=s\sqrt{\e}$, condition \eqref{eq.lambdaxi}.
Finally, for $\pi=\pm1$ we have
$|\pi s+\xi|\ge s/2$, thus
$$
t=
\Bigl\|\frac{\pi se+\lambda}{\pi s+\xi}-e\Bigr\|
=\Bigl\|\frac{\lambda-\xi e}{\pi s+\xi}\Bigr\|
\le\frac{2r}{s/2}=
\Delta/4,
$$
and so applying \eqref{eq.normdif} for
$h=\Bigl(\frac{\lambda-\xi e}{\pi s+\xi}\Bigr)/t$
we get
$$
\Bigl\|\frac{\pi se+\lambda}{\pi s+\xi}\Bigr\|
\le
1+e^*\Bigl(\frac{\lambda-\xi e}{\pi s+\xi}\Bigr)
+\eta|t|.
$$
Note that
$e^*\Bigl(\frac{\lambda-\xi e}{\pi s+\xi}\Bigr)=0$
as $e^*(\lambda)=re^*(u)=\xi=e^*(\xi e)$ and
hence \eqref{eq.estup}:
$$
\Bigl\|\frac{\pi se+\lambda}{\pi s+\xi}\Bigr\|
\le
1+\eta\Delta/4
=
1+\e/4.
$$

Define $u_1=-e$, $u_2=e$ and $u_3=({r}/{s})u$.
Note that $r/s=\Delta/16<1$, thus all vectors
$u_1,u_2,u_3$ are in the unit ball.
We also have $s<16\delta/\Delta=\frac{1}{2}\rho\sqrt{\Delta\eta}
<\rho<\delta_*(\Delta\eta,\Delta^2\eta^2/320)$, and therefore
as  $F_\e$ is a $\delta\theta$-wedge approximation for
$\ball{y}{s}$,
we can find $u_1',u_2',u_3'$
such that $\|u_i-u_i'\|<\theta={\Delta^2\eta^2}/{320}$
for $i=1,2,3$ and
$$
[y-s_1,y+\lambda']\cup[y+\lambda',y+s_2']\subseteq F_{\e},
$$
where $s_1=-su_1'$, $s_2=su_2'$ and $\lambda'=su_3'$.
We then have $\|s_i-se\|=s\|u_i-u_i'\|\le{\Delta^2\eta^2}s/{320}$
for $i=1,2$ and
$\|\lambda'-\lambda\|=
\|su_3'-ru\|=
s\|u_3'-u_3\|
\le{\Delta^2\eta^2}s/{320}<{\Delta\eta}s/{16}$, which verifies
\eqref{eq.ests1s2}
and \eqref{eq.estlam}.

Therefore all conditions of Lemma~\ref{lem.max}
are satisfied; hence we may find
$$
y'\in [y-s_1,y+\lambda']\cup[y+\lambda',y+s_2]\subseteq F_{\e}
$$ and a direction $e'\in S(Y)$ for which $f(y',e')$ exists, with
$f'(y',e')\ge f'(y,e)+\e$, and such that
for all $t\in\real$ the inequality \eqref{eq.condi} is satisfied.
But for every pair $(y',e')$ from $F_{\e}\times S(Y)$ that satisfies
\eqref{eq.condi}
we have  $f'(y',e')< f'(y,e)+\e$,
a contradiction.

Hence for every $r\in(0,\delta)\text{ and }\|h\|\le1$,
\eqref{difR} is satisfied.
\end{proof}

\section{Optimisation}
\label{Optimisation}
In this section we prove Theorem~\ref{th.incr}.
It describes how, given a Lipschitz function $g$ on a Banach space $Y$
and a bundle $\pi:M\to Y$, where $(M,d)$ is a complete metric space and
$\pi$ is continuous,
one finds a point $\tilde x \in M$ and direction $\tilde e$ in the unit sphere of $Y$ with almost locally maximal directional derivative.

We describe how to choose the desired sequence of pairs
$$
(x_n,e_n)_{n\ge0}\subseteq M\times S(Y)
$$
as an inductive
procedure. While convergence of $(x_n)_{n\ge0}$ simply follows from
the fact that $x_{n+1}$ is chosen very close to $x_n$, we shall need
additional work in order to obtain the convergence of $e_n$. For this,
we change the norm on each step; see \eqref{eq.pn} and Lemma~\ref{lemleq}.
We then argue in Section~\ref{sec.diff}
that the sequence of norms defined in \eqref{eq.pn}
converges to the norm $\|\cdot\|'$ specified in Theorem~\ref{th.incr}.

Suppose the assumptions of Theorem~\ref{th.incr} are satisfied.
We thus have a Lipschitz function $g$ acting on a Banach space $Y$
such that the set
$$
D=\{(x,e)\in M\times \left(Y\setminus\{0\}\right):
\text{the directional derivative }g'(\pi x,e) \text{ exists}\}
$$
is not empty.
Assume without loss of
generality that $\mathrm{Lip}(g) = 1/3$.

Recall $\|e_0\| = 1$ and
$g'(\pi x_0,e_0) \geq 0$.
Choose $e_0^* \in Y^*$ with $e_0^*(e_0) = 1$ and $\|e_0^*\| = 1$, and
define
\begin{equation}\label{eq.deff-g}
f = g + \frac{2}{3}e_0^*
\end{equation}
so that item (1) of Theorem~\ref{th.incr} is satisfied.
Note that $f-g$ is linear, so $f$ is a Lipschitz function with
$\mathrm{Lip}(f) \leq 1$.
As $f-g$ is linear, the set $D$ is precisely the set of all
$$(x,e) \in M \times (Y\setminus \{0\})$$
such that $f'(\pi x,e)$ exists.

We can immediately make a very simple observation:
if $f'(\pi x_0,e_0) \leq f'(\pi x,e)$ then
$$g'(\pi x_0,e_0) + \frac{2}{3} \leq g'(\pi x,e) + \frac{2}{3} e_0^*(e),$$
so that
\begin{equation}\label{eq.start}
e_0^*(e)\ge1-\frac{3}{2}g'(\pi x,e)\ge\frac{1}{2}.
\end{equation}

Note that for any Lipschitz function $f:Y\to\real$  with
$\mathrm{Lip}(f) \leq  1$ and $x,x' \in M$, $e \in Y$ with $\|e\| \leq 1$, we have
$$
|(f(\pi x'+te)-f(\pi x'))-(f(\pi x+te)-f(\pi x))| \leq 2|t|;
$$
therefore, we may assume that
$\Theta(t) \leq 2$ for all $t > 0$.

We now introduce a function
$\Omega(t):(0,\infty)\to(0,\infty)$ that we are going to use instead of
$\Theta(t)$ in our subsequent argument.

\begin{lemma}
\label{omegalemma}
If $\Theta:(0,\infty)\to(0,2]$ satisfies $\Theta(t)\to0$ as $t\to0$
then
there exists a function $\Omega\colon (0,\infty) \to (0,\infty)$ such that
\begin{enumerate}
\item[(1)]
$\Omega(t) \geq 2\Theta(t)$ for all $t \in \R$,
\item[(2)]
$\Omega(t) \rightarrow 0$ as $t \rightarrow 0^+$,
\item[(3)]
if $A,B > 0$ then $\Omega(A) + 2B \leq \Omega(A+B)$.
\end{enumerate}
\end{lemma}
\begin{proof}
For each $n \in \Z$, define
$\beta(2^n):= \sup_{0 < t' \leq 2^{n+1}} \Theta(t')$.
We may uniquely extend $\beta$ to $(0,\infty)$ by imposing the property that $\beta$ is affine on each interval of the form $[2^n,2^{n+1}]$ for $n \in \Z$. Note that $\beta$ is continuous, increasing and $\beta(t) \geq \Theta(t)$ for every $t > 0$.
Further for $t \leq 2^n$ where $n \in \Z$, we have
$\beta(t) \leq \beta(2^{n}) = \sup_{0 < t' \leq 2^{n+1}} \Theta(t')$
and as $\Theta(t) \to 0$ as $t \to 0^+$ we deduce that $\beta(t) \to 0$ as $t \to 0^+$.

We now let $\Omega(t) = 2\beta(t)+2t$. Then (1) and (2) are immediate
as $\beta(t) \geq \Theta(t)$ and $\beta(t) \to 0$ as $t \to 0^+$.
Finally for (3) we may use the fact that $\beta$ is increasing to deduce that for $A,B > 0$,
$\Omega(A+B)= 2\beta(A+B)+2A+2B\geq 2\beta(A)+2A+2B
= \Omega(A)+2B$.
\end{proof}

We now  define a notion of weight and a class of pairs
that weigh more than the given pair.
\begin{definition}
\label{defnweight}
If $p$ is a norm on $Y$ and $(x,e) \in D$ then we call
$$
w_{p}(x,e) = \frac{f'(\pi x,e)}{p(e)}
$$
the weight of $(x,e)$ with respect to the norm $p$.

Further for $\sigma \geq 0$ we let $G_{p}(x,e,\sigma)$ be the set of all $(x',e') \in D$ such that
\begin{equation}\label{eq.defg1}
w_{p}(x,e) \leq w_{p}(x',e')
\end{equation}
and
\begin{align}
\notag
|(f(\pi x'&+te)-f(\pi x'))-(f(\pi x+te)-f(\pi x))| \\
\label{eq.defg2}
&\leq \left(\sigma + \Omega(w_{p}(x',e')-w_{p}(x,e))\right)|t|
\end{align}
for all $t \in \R$, where the function $\Omega$ is given by Lemma~\ref{omegalemma}.
\end{definition}

In what follows, the notation $\|y- \R e\|$ where $y\in Y$ and
$e\in Y\setminus\{0\}$ is used for the distance between the point $y$
and the one dimensional subspace of $Y$ generated by $e$.
This distance is calculated with the original norm $\| \cdot \|$ on $Y$.

\subsection{Inductive construction}
Let $\sigma_0=16$, $\delta_0 = 1$, $t_{0} \in (0,1/2)$,
the norm $p_0=\| \cdot \|$ and
$w_0 = w_{p_0}$. The pair $(x_0,e_0)$ was chosen earlier.
Below we will define
various positive parameters $\sigma_n,t_n,\e_n,\nu_n,\Delta_n,\delta_n$,
nested sequence $D_n$ of non-empty subsets of $D$ and
pairs $(x_n,e_n)\in D_n$.
For every $n\ge1$, we define
\begin{equation}\label{eq.pn}
p_n(y) = \sqrt{\|y\|^2 + \sum_{m=0}^{n-1} t_m^2 \|y-\R e_m\|^2}
\end{equation}
and let $w_n=w_{p_n}$ be the
weight function defined on $D$. It is clear
\eqref{eq.pn} defines a norm on $Y$
and $p_n(y)\ge\max\{\|y\|,p_{n-1}(y)\}$ for all $y\in Y$.
Together with
$\mathrm{Lip}(f) \leq 1$, this implies
$w_n(x,e) \leq \min\{1,w_{n-1}(x,e)\}$ for any $(x,e) \in D$.

For every $n\ge1$, choose
\begin{equation}\label{eq.param}
\sigma_{n} \in (0,\sigma_{n-1}/16),
t_{n} \in (0,t_{n-1}/2)\textrm{ with }t_n^2 < \sigma_{n-1}/16
\textrm{ and }
\e_{n} \in (0,t_n^2\sigma_n^2/2^{13}).
\end{equation}
Let $D_{n}$ to be the set of all pairs $(x,e) \in D$ with $d(x,x_{n-1}) < \delta_{n-1}$, $\|e\| = 1$ and
$$(x,e) \in G_{p_n}(x_{n-1},e_{n-1},\sigma_{n-1}-\nu)$$
for some $\nu \in (0,\sigma_{n-1}/2)$.
Note that $(x_{n-1},e_{n-1})\in D_n$, and so $D_n\ne\emptyset$.
Since $w_n$ is bounded by $1$ from above we can
choose $(x_{n},e_{n}) \in D_{n}$ such that
for every $(x,e)\in D_n$
\begin{equation}\label{eq.defxnen}
w_n(x,e) \leq w_n(x_n,e_n)+\e_n.
\end{equation}
Note that the definition of $D_n$ then implies
$d(x_n,x_{n-1})<\delta_{n-1}$, and as $(x_n,e_n)\in D_n$ and
$p_n(e_{n-1})=p_{n-1}(e_{n-1})$,
we have for every $n\ge1$
\begin{equation}
\label{eq.co5}
w_{n-1}(x_{n-1},e_{n-1})=
w_n(x_{n-1},e_{n-1})\le w_n(x_n,e_n).
\end{equation}
This implies $w_n(x,e)\ge w_0(x_0,e_0)=f'(\pi x_0,e_0)$ for
every $(x,e)\in D_n$; in particular, \eqref{eq.start} implies
\begin{equation}
\label{eq.e0dn}
e_0^*(e)\ge1/2
\end{equation}
for any $(x,e)\in D_n$.

Let $\nu_{n} \in (0,\sigma_{n-1}/2)$ be such that $(x_n,e_n) \in G_{p_n}(x_{n-1},e_{n-1},\sigma_{n-1}-\nu_{n})$.
Pick $\Delta_n > 0$ such that
\begin{align}\label{wot1}
|f(\pi x_{n}+te_{n})-f(\pi x_{n})-f'(\pi x_n,e_n)t| &\leq \sigma_{n-1}|t|/32\\
\label{wot2}
|f(\pi x_{n-1}+te_{n-1})-f(\pi x_{n-1})-f'(\pi x_{n-1},e_{n-1})t| &\leq \sigma_{n-1}|t|/32
\end{align}
for all $t$ with $|t| \leq 4\Delta_{n}/\nu_{n}$.

Finally choose
$\delta_{n} \in (0,(\delta_{n-1} - d(x_n,x_{n-1}))/2)$ such that
$\|\pi x-\pi x_n\| \leq \Delta_n$ whenever $d(x,x_n) \leq \delta_n$; such a $\delta_n$ exists because $\pi$ is continuous.

Let us make some simple observations.
First of all, \eqref{eq.param} implies that the sequences
$\sigma_n,t_n,\e_n$ all tend to zero. Since $\nu_n<\sigma_{n-1}/2$ and
$\delta_n<(\delta_{n-1}-d(x_n,x_{n-1}))/2$ we conclude that
$\nu_n$ and $\delta_n$ tend to
zero, too. The latter inequality also implies
\begin{equation}\label{eq.balls}
\overline{B}_{\delta_n}(x_{n}) \subseteq \ball{x_{n-1}}{\delta_{n-1}}
\end{equation}
for every $n\ge1$ and so
\begin{equation}\label{convofxn}
d(x_k,x_n)<\delta_n \text{ for all $k\ge n$.}
\end{equation}
Since $M$ is complete
we conclude that the sequence $(x_n)$ converges in $M$ to some
point $\limx$.

The inequality $t_n<t_{n-1}/2$ also implies
$p_n(y)^2  \leq \|y\|^2 + 2t_0^2 \cdot \|y\|^2
\leq 2 \|y\|^2$, so for all $y\in Y$,
\begin{equation}\label{eq.normineq}
\|y\|\le p_n(y)\le2\|y\|.
\end{equation}
Then, using $p_n(e_{n-1}) \leq 2$, we get for every $(x,e)\in D$
\begin{align}
\notag&|f'(\pi x,e)-f'(\pi x_{n-1},e_{n-1})|
\leq 2\frac{|f'(\pi x,e)-f'(\pi x_{n-1},e_{n-1})|}{p_{n}(e_{n-1})}\\
\notag&\leq 2\left|\frac{f'(\pi x,e)}{p_{n}(e)}-\frac{f'(\pi x_{n-1},e_{n-1})}{p_{n}(e_{n-1})}\right|
+ 2|f'(\pi x,e)|\left|\frac{1}{p_{n}(e_{n-1})}-\frac{1}{p_{n}(e)}\right|\\
\notag&\leq 2|w_n(x,e)-w_{n}(x_{n-1},e_{n-1})| + 2\frac{\|e\|}{p_n(e_{n-1})p_n(e)} |p_n(e)-p_n(e_{n-1})|\\
\label{eq.fpin}&\leq 2|w_n(x,e)-w_{n}(x_{n-1},e_{n-1})| + 4 \|e-e_{n-1}\|,
\end{align}
where, in the penultimate line, we are using $\mathrm{Lip}(f) \leq 1$ and, in the final line, $p_n(e) \geq \|e\|$,
$p_n(e_{n-1})\ge\|e_{n-1}\|=1$ and the fact that
$$
|p_n(e)-p_n(e_{n-1})| \leq p_n(e-e_{n-1}) \leq 2\|e-e_{n-1}\|.
$$

We are now ready to prove a very important property of sets
$D_n$; the ``moreover'' part of Lemma~\ref{lemleq}
together with \eqref{eq.param} implies the
convergence of the sequence $(e_n)$
to some $\lime\in Y$ with $\|\lime\|=1$. We will show later that the pair
$(\limx,\lime)$ has the properties required by
Theorem~\ref{th.incr}.

\begin{lemma}
\label{lemleq}
For every $n\ge1$, we have the inclusions
$D_{n+1}\subseteq  G_{p_n}(x_{n-1},e_{n-1},\sigma_{n-1}-\nu_n/2)$
and $D_{n+1} \subseteq D_{n}$.
Moreover, for any $(x,e) \in D_{n+1}$
we have $\|e-e_{n}\| \leq \sigma_{n}/8$.
\end{lemma}
\begin{proof}
Notice first that since $\sigma_0=16$,
the ``moreover'' statement is satisfied for $n=0$.

We shall now show that assuming the latter statement is satisfied for $n-1$,
where $n\ge1$, the full conclusion of the present lemma
holds for $n$.

Assume therefore $n\ge1$,
the ``moreover'' part is satisfied for $n-1$ and
$(x,e) \in D_{n+1}$. Since $(x_{n},e_{n}) \in D_{n}$,
we get
\begin{equation}\label{enen-1}
\|e_{n}-e_{n-1}\|
\leq
\frac{\sigma_{n-1}}{8}.
\end{equation}

Since
$(x,e)\in G_{p_{n+1}}(x_n,e_n,\sigma_n-\nu)$ for some $\nu>0$,
we get $w_{n+1}(x,e) \geq w_{n+1}(x_n,e_n)$; thus
using \eqref{eq.co5}, we obtain
the first defining property of $G_{p_n}(x_{n-1},e_{n-1},*)$:
\begin{equation*}
w_n(x,e) \geq w_{n+1}(x,e) \geq w_{n+1}(x_n,e_n) = w_n(x_n,e_n) \geq w_{n}(x_{n-1},e_{n-1}).
\end{equation*}
In order to show $(x,e)\in G_{p_n}(x_{n-1},e_{n-1},\sigma_{n-1}-\nu_n/2)$,
we need to prove the second defining property of the latter set.
We prove the inequality separately for $|t|<4\Delta_n/\nu_n$ and
$|t|\ge4\Delta_n/\nu_n$.

If $|t|<4\Delta_n/\nu_n$, using first~\eqref{wot1},~\eqref{wot2} and then $\mathrm{Lip}(f) \leq 1$,
\begin{align*}
|(f(\pi x+te_{n-1})-&f(\pi x))-(f(\pi x_{n-1}+te_{n-1})-f(\pi x_{n-1}))|\\
\leq |(f(\pi x&+te_{n-1})-f(\pi x))-(f(\pi x_n+te_n)-f(\pi x_n))|\\
&+ |f'(\pi x_n,e_n)-f'(\pi x_{n-1},e_{n-1})|\cdot |t| + \frac{1}{16}\sigma_{n-1}|t|\\
\leq |(f(\pi x&+te_n)-f(\pi x))-(f(\pi x_n+te_n)-f(\pi x_n))|  +  \| e_{n}-e_{n-1}\|\cdot |t|\\
&+ |f'(\pi x_n,e_n)-f'(\pi x_{n-1},e_{n-1})|\cdot |t| + \frac{1}{16}\sigma_{n-1}|t|.
\end{align*}
We may now apply \eqref{enen-1}, $(x,e)\in G_{p_{n+1}}(x_n,e_n,\sigma_n-\nu)$ and
\eqref{eq.fpin} to deduce that the latter is bounded from above by
\begin{multline}
|t|\Bigl(\sigma_n-\nu+\Omega (w_{n+1}(x,e) - w_{n+1}(x_n,e_n)) + \frac{3}{16}\sigma_{n-1} +\Bigr.\\
\Bigl.2(w_n(x_n,e_n)-w_{n}(x_{n-1},e_{n-1}) + 4\|e_n-e_{n-1}\|\Bigr).
\end{multline}
Recall that $\Omega$ is an increasing function and
$$
w_{n+1}(x,e) - w_{n+1}(x_n,e_n)
=
w_{n+1}(x,e) - w_{n}(x_n,e_n)
\le
w_{n}(x,e) - w_{n}(x_n,e_n);
$$
then using again \eqref{enen-1},
 $\sigma_n \in (0,\sigma_{n-1}/16)$, $\nu_n \in (0,\sigma_{n-1}/2)$
so that $\frac{3}{4}\sigma_{n-1}\le\sigma_{n-1}-\nu_n/2$,
and Lemma~\ref{omegalemma}(3),
we have
\begin{align*}
&|(f(\pi x+te_{n-1})-f(\pi x))-(f(\pi x_{n-1}+te_{n-1})-f(\pi x_{n-1}))|\\
\leq
&\left(\frac{3}{4}\sigma_{n-1}+\Omega (w_{n}(x,e) - w_{n}(x_n,e_n)) + 2(w_n(x_n,e_n)-w_{n}(x_{n-1},e_{n-1})) \right)|t|\\
\leq
&\left(\sigma_{n-1}-\nu_n/2 +\Omega (w_{n}(x,e) - w_{n}(x_{n-1},e_{n-1}))\right)|t|.
\end{align*}

Now we consider the case $|t| \geq 4\Delta_n/\nu_n$. As
$(x,e)\in D_{n+1}$, we have $d(x,x_n) < \delta_n$. Therefore,
from  the definition of $\delta_n$, we have
$\|\pi x-\pi x_n\| \le \Delta_n \le \nu_n|t|/4$. Thus,
replacing $f(\pi x+te_{n-1})$ with $f(\pi x_n+te_{n-1})$
and $f(\pi x)$ with $f(\pi x_n)$, we get
\begin{multline*}
|(f(\pi x+te_{n-1})-f(\pi x))-(f(\pi x_{n-1}+te_{n-1})-f(\pi x_{n-1})|\\
 \leq \nu_n|t|/2 + |(f(\pi x_n+te_{n-1})-f(\pi x_n))-(f(\pi x_{n-1}+te_{n-1})-f(\pi x_{n-1})|.
\end{multline*}
Now using
$(x_n,e_n) \in G_{p_n}(x_{n-1},e_{n-1},\sigma_{n-1}-\nu_{n})$,
we estimate the second term by
$$(\sigma_{n-1}-\nu_n + \Omega(w_n(x_n,e_n)-w_n(x_{n-1},e_{n-1})))|t|.$$
Adding  $\nu_n|t|/2$ to this and noting $\Omega$ is an increasing
function, we estimate this from above by
$$
\left(\sigma_{n-1}-\nu_n/2 + \Omega(w_n(x,e)-w_n(x_{n-1},e_{n-1}))\right)|t|.
$$
This finishes the proof of $(x,e)\in G_{p_n}(x_{n-1},e_{n-1},\sigma_{n-1}-\eta_n/2)$.

Further, for $(x,e) \in D_{n+1}$ we have $\|e\| = 1$ and
$d(x,x_{n-1}) < \delta_{n-1}$, using $d(x,x_n) < \delta_n$ and
\eqref{eq.balls}. Therefore, $(x,e) \in D_{n}$; hence
$D_{n+1}\subseteq D_n$.

Finally to prove $\|e-e_n\|\le\sigma_n/8$,
note that \eqref{eq.co5} together with the definition of
$(x_n,e_n)$ implies
\begin{equation}
\label{wotname}
w_n(x_n,e_n) \leq
w_{n+1}(x_{n},e_{n}) =\frac{p_n(e)}{p_{n+1}(e)} w_n(x,e)
 \leq \frac{p_n(e)}{p_{n+1}(e)} (w_n(x_n,e_n)+\e_n).
\end{equation}

Writing $p_{n+1}(e)=\sqrt{p_n^2(e)+t_n^2d^2}$, where
$d = \|e-\R e_n\| \leq 1$ and using  $t_n < t_0 < 1/2$ we deduce
$$
\frac{p_n(e)}{p_{n+1}(e)} = {1}/{\sqrt{1+t_n^2 d^2/p_n(e)^2}} \leq 1-\frac{t_n^2 d^2}{4p_n(e)^2}
$$
as $1/\sqrt{1+x} \leq 1-x/4$ for $0 \leq x \leq 1$.
Substituting this inequality into~\eqref{wotname}
and using \eqref{eq.param} we obtain
$$
\frac{t_n^2 d^2}{4p_n(e)^2}w_n(x_n,e_n) \leq \e_n \left(1-\frac{t_n^2 d^2}{4p_n(e)^2}\right)\leq \e_n
<t_n^2\sigma_n^2/2^{13}.$$
On the other hand, \eqref{eq.co5} and $g'(\pi x_0,e_0) \geq 0$
imply
$$
w_n(x_n,e_n)\ge w_0(x_0,e_0)=f'(\pi x_0,e_0)
= g'(\pi x_0,e_0)+\frac{2}{3} > \frac{1}{2},
$$
so using $p_n(e)\le2$ we conclude $d\le\sigma_n/2^4$.
This means there is a $t \in \R$ such that
\begin{equation}
\label{runoutof}
\|e-t e_n\| \leq \frac{\sigma_n}{16}.
\end{equation}
It follows $|e_0^*(e-t e_n)| \leq \sigma_n/16 \leq 1/2$. However,
by \eqref{eq.e0dn},
$e_0^*(e), e_0^*(e_n)\geq 1/2$, hence $t \geq 0$.
Then from~\eqref{runoutof} and $\|e_n\| = \|e\| = 1$ we get that
$|1-t| \leq \frac{\sigma_n}{16}$
and so
$$
\|e-e_n\| \leq \frac{\sigma_n}{8}.\qedhere
$$
\end{proof}

We note here that Lemma~\ref{lemleq} implies that $\|e_m-e_n\|\le\sigma_n/8$
whenever $m\ge n+1$. Thus $(e_n)$ is a Cauchy sequence, so it converges.
Let $\lime=\lim e_n$. As $\|e_n\|=1$ for each $n\ge1$, we have $\|\lime\|=1$.

\subsection{Existence of directional derivative
$f'(\pi\limx,\lime)$}\label{sec.diff}
From~\ref{convofxn} we have $d(x_k,x_n)<\delta_n$ for all $k\ge n$. We also know that for $k \geq n$, $(x_k,e_k) \in D_{k+1} \subseteq D_{n+1}$ using Lemma~\ref{lemleq}, so that $\|e_k-e_n\| \leq \sigma_{n}/8$, again by Lemma~\ref{lemleq}. Hence the sequences $x_{n}$ and $e_{n}$ converge to $\limx$
and $\lime$ respectively, where
\begin{equation}\label{eq.distinf}
d(\limx,x_n)<\delta_n
\textrm{ and }
\|\lime-e_n\|\le\sigma_n/8
\end{equation}
are satisfied for every $n\ge1$, the strictness of the first inequality following from \eqref{eq.balls}.
It is also
clear that the sequence of norms $p_n$ converges to
$$
p_\infty(y)= \sqrt{\|y\|^2 + \sum_{m=1}^{\infty} t_m^2 \|y-\R e_m\|^2}
$$
as this formula defines a norm and
$$
p_n^2(y) \leq p_{\infty}^2(y)  \leq
p_n^2(y)+2t_n^2\|y\|^2\le (1+2t_n^2)p_n^2(y)
 \leq (1+t_n^2)^2 p_n^2(y)
$$
implies for all $y\in Y$
\begin{equation}\label{eq.pminf}
p_n(y)\le p_\infty(y)\le(1+t_n^2)p_n(y).
\end{equation}
This implies for every $(x,e)\in D$
\begin{equation}\label{eq.wnwinf}
|w_n(x,e)-w_{\infty}(x,e)|
= |f'(x,e)| \cdot \frac{|p_{\infty}(e)-p_n(e)|}{p_n(e)p_{\infty}(e)}
\leq \|e\| \frac{t_n^2 \cdot p_n(e)}{p_n(e)p_{\infty}(e)}
\leq t_n^2
\end{equation}
using $\mathrm{Lip}(f) \leq 1$ and $p_{\infty}(e) \geq \|e\|$.

We will now show that the directional derivative
$f'(\pi\limx,\lime)$ exists and
\begin{equation}\label{eq.wnincr}
w_m(x_m,e_m) \nearrow w_{\infty}(\limx,\lime),
\end{equation}
where $w_\infty=w_{p_\infty}$.

Indeed, for every $n\ge1$, the inequality $p_n(y)\ge\|y\|$ and
\eqref{eq.co5} imply
$$
0<w_0(x_0,e_0)\le w_n(x_n,e_n)\le\mathrm{Lip}(f)\le1.
$$
Thus there is
$L\in(0,1]$ such that $w_n(x_n,e_n)\nearrow L$. From \eqref{eq.pminf}
we conclude $w_\infty(x_n,e_n)\to L$ and
$w_{n+1}(x_{n},e_{n}) \to L$.
Note then
$$
w_{m}(x_{n},e_{n})-w_{m}(x_{m-1},e_{m-1})
\xrightarrow[n\to\infty]{}\frac{p_{\infty}(\lime)}{p_m(\lime)} L - w_m(x_{m-1},e_{m-1})
=:s_{m}
\xrightarrow[m\to\infty]{} 0.
$$
Assuming $n\ge m$ we get $(x_n,e_n)\in D_n\subseteq D_{m+1}$.
The first condition \eqref{eq.defg1} of
\begin{equation}\label{eq.uniformeps}
(x_n,e_n)\in G_{p_m}(x_{m-1},e_{m-1},\sigma_{m-1}-\nu_m/2)
\end{equation}
says
$w_m(x_n,e_n)\ge w_m(x_{m-1},e_{m-1})$, thus
$s_{m} \geq 0$ for each $m$.
Taking $n \rightarrow \infty$ in the second inequality \eqref{eq.defg2} from
the definition of
\eqref{eq.uniformeps}, we obtain
\begin{equation}\label{eqfm}
|(f(\pi\limx+te_{m-1})-f(\pi\limx))-(f(\pi x_{m-1}+te_{m-1})-f(\pi x_{m-1}))|
\leq r_{m}|t|
\end{equation}
for any $t \in \R$,
where
\begin{equation*}
r_{m}:= \sigma_{m-1} - \nu_{m}/2 + \Omega(s_m) \to 0
\end{equation*}
by Lemma~\ref{omegalemma}(2).
Using $\|\lime-e_{m-1}\| \leq \sigma_{m-1}$ and $\mathrm{Lip}(f) \leq 1$:
\begin{equation}\label{eqfmmod}
|(f(\pi\tilde  x+t\lime)-f(\pi\tilde  x))-(f(\pi x_{m-1}+te_{m-1})-f(\pi x_{m-1}))| \leq (r_m+\sigma_{m-1})|t|.
\end{equation}
Let $\e > 0$. Note that as
$$f'(\pi x_{m-1},e_{m-1}) = p_{m-1}(e_{m-1})w_{m-1}(x_{m-1},e_{m-1}) \to p_{\infty}(\lime)L$$
we may pick $m$ such that
\begin{equation}\label{mlarge}
r_m+\sigma_{m-1} \leq \e/3
\quad\textrm{ and }\quad
|f'(\pi x_{m-1},e_{m-1})-p_{\infty}(\lime)L| \leq \e/3
\end{equation}
and then $\delta > 0$ with
\begin{equation}\label{dirderivatm}
|f(\pi x_{m-1}+te_{m-1})-f(\pi x_{m-1})-f'(\pi x_{m-1},e_{m-1})t| \leq \e |t|/3
\end{equation}
for all $t$ with $|t| \leq \delta$. Combining \eqref{eqfmmod},
\eqref{mlarge} and \eqref{dirderivatm} we obtain
\begin{equation*}
|f(\pi\tilde  x+t\lime)-f(\pi\tilde  x)-p_{\infty}(\lime)Lt| \leq \e |t|
\end{equation*}
for $|t| \leq \delta$. Hence the directional derivative
$f'(\pi\tilde  x,\lime)$ exists and equals
$p_{\infty}(\lime)L$ and $w_\infty(\limx,\lime)=L$.

The last equality and the definition of $s_m$ implies
$w_m(\limx,\lime)-w_m(x_{m-1},e_{m-1})=s_m\ge0$, so together with
\eqref{eqfm} we get
\begin{equation}\label{eq.xinfg}
(\limx,\lime)\in G_{p_m}(x_{m-1},e_{m-1},\sigma_{m-1}-\nu_m/2)
\end{equation}
and so $(\limx,\lime)\in D_m$ for all $m\ge1$.

\subsection{Maximality of the weight function at $(\limx,\lime)$}
\label{localmaximality}
We now verify that the value of the weight
function $w_\infty(\limx,\lime)$ is almost maximal in the
following sense:
For every $\e > 0$ there exists $\delta > 0$ such that whenever
$(x',e') \in G_{p_{\infty}}(\limx,\lime,0)$
with $d(x',\limx) \leq \delta$ then we have
\begin{equation}\label{eq.maxw}
w_{\infty}(x',e') < w_{\infty}(\limx,\lime) + \e.
\end{equation}

Assume $\e>0$ is fixed, choose then
$n \geq 1$ with $\e_n + 2t_n^2 < \e$ and pick $\Delta > 0$ such that for $|t| < 8\Delta/\nu_n$, the following two inequalities are satisfied:
\begin{align}
\label{wot12}
|f(\pi \limx+t\lime)-f(\pi \limx)-f'(\pi \limx,\lime)t| \leq &\frac{1}{16}\sigma_{n-1}|t|;\\
\label{wot22}
|f(\pi x_{n-1}+te_{n-1})-f(\pi x_{n-1})-f'(\pi x_{n-1},e_{n-1})t| \leq &\frac{1}{16}\sigma_{n-1}|t|.
\end{align}

Using \eqref{eq.distinf}  and the continuity of $\pi$ we can find
\begin{equation}\label{delta2}
\delta \in (0,\delta_{n-1}-d(\limx,x_{n-1}))
\end{equation}
such that whenever $d(x',\limx) \leq \delta$,
\begin{equation}
\label{hahamissed}
\|\pi x'-\pi \limx \| \leq \Delta.
\end{equation}

We now suppose, for a contradiction, that we may find
$(x',e') \in G_{p_{\infty}}(\limx,\lime,0)$
such that $d(x',\limx) \leq \delta$ and, contrary to \eqref{eq.maxw},
we have
\begin{equation}\label{choicex'}
w_{\infty}(x',e') \geq w_{\infty}(\limx,\lime) + \e.
\end{equation}
As $w_{\infty}(x',e')$ is invariant if we scale $e'$ by a positive factor, as is the membership relation $(x',e') \in G_{p_{\infty}}(\limx,\lime,0)$, we may assume that $\|e'\| = 1$.

First we shall show that $(x',e') \in D_n$.
Since \eqref{delta2} and $d(x',\limx) \leq \delta$ imply that we have
$d(x',x_{n-1})< \delta_{n-1}$,
by definition of $D_n$ to prove $(x',e') \in
D_{n}$ it is enough to show that
\begin{equation}\label{eqlast2}
(x',e') \in G_{p_n}(x_{n-1},e_{n-1},\sigma_{n-1}-\nu_n/4).
\end{equation}

Note that
from \eqref{eq.wnwinf} we have
\begin{equation}\label{eq.wnwx'}
w_n(x',e') -w_n(\limx,\lime)
\ge
w_\infty(x',e') -w_\infty(\limx,\lime)-2t_n^2
\ge
\e-2t_n^2\ge \e_n>0;
\end{equation}
therefore
$$
w_n(x',e') >w_n(\limx,\lime)
\ge w_n(x_{n-1},e_{n-1})
$$
as $s_n=w_n(\limx,\lime)-w_n(x_{n-1},e_{n-1})\ge0$: see the end of Section~\ref{sec.diff}.

We now check the second condition of \eqref{eqlast2}.
Assume $|t|<8\Delta/\nu_n$; using~\eqref{wot12},~\eqref{wot22} and then $\mathrm{Lip}(f) \leq 1$,
\begin{align*}
|(f(\pi x'&+te_{n-1})-f(\pi x'))-(f(\pi x_{n-1}+te_{n-1})-f(\pi x_{n-1}))|\\
\leq & \text{ }|(f(\pi x'+te_{n-1})-f(\pi x'))-(f(\pi \limx+t\lime)-f(\pi \limx))|\\ &+ |f'(\pi \limx,\lime)-f'(\pi x_{n-1},e_{n-1})| \cdot |t| + \frac{1}{8}\sigma_{n-1}|t|\\
\leq & \text{ }|(f(\pi x'+t\lime)-f(\pi x'))-(f(\pi \limx+t\lime)-f(\pi \limx))|  + \|\lime-e_{n-1}\|\cdot |t|\\&+ |f'(\pi \limx,\lime)-f'(\pi x_{n-1},e_{n-1})|\cdot |t| + \frac{1}{8}\sigma_{n-1}|t|.
\end{align*}
In this sum of four terms, we use $(x',e') \in G_{p_{\infty}}(\limx,\lime,0)$ to
bound the first term from above by
$\Omega(w_{\infty}(x',e')-w_{\infty}(\limx,\lime))\cdot|t|$,
and \eqref{eq.fpin} to
bound
the third term by
$$\left(2(w_n(\limx,\lime)-w_n(x_{n-1},e_{n-1})) + 4\|\lime-e_{n-1}\|\right)|t|.$$
Using
in addition inequality
$\|\lime-e_{n-1}\| \leq \sigma_{n-1}/8$ from
\eqref{eq.distinf}, we get

\begin{multline}\label{eq.estx'}
|(f(\pi x'+te_{n-1})-f(\pi x'))-(f(\pi x_{n-1}+te_{n-1})-f(\pi x_{n-1})|\\
\le
|t|\left(
\Omega(w_{\infty}(x',e')-w_{\infty}(\limx,\lime))
+2(w_n(\limx,\lime)-w_n(x_{n-1},e_{n-1}))
+\frac{3}{4}\sigma_{n-1}\right)
\end{multline}
We now use the fact that $\Omega$ is an increasing function
and $w_\infty(x',e')\le w_n(x',e')$, which follows from
$p_n\le p_\infty$, and then use
\eqref{eq.wnwinf} to
estimate
$2w_n(\limx,\lime)$ from above
by the expression
$2w_\infty(\limx,\lime)+2t_n^2$.
Then the expression on the right hand side of \eqref{eq.estx'} is
less than or equal to
$$
|t|\Bigl(
\Omega(w_n(x',e')-w_{\infty}(\limx,\lime))
+2(w_\infty(\limx,\lime)-w_n(x_{n-1},e_{n-1}))
+\frac{3}{4}\sigma_{n-1}+2t_n^2\Bigr).
$$
As $t_n^2<\sigma_{n-1}/16$ and $\Omega$ satisfies property (3) in
Lemma~\ref{omegalemma}, we finally get
\begin{multline*}
|(f(\pi x'+te_{n-1})-f(\pi x'))-(f(\pi x_{n-1}+te_{n-1})-f(\pi x_{n-1})|\\
\le
\left(
\Omega (w_n(x',e')-w_n (x_{n-1},e_{n-1}))
+\frac{7}{8}\sigma_{n-1}\right)|t|\\
\le
\Bigl(\sigma_{n-1}-\nu_n/4 +\Omega (w_{n}(x',e') - w_{n}(x_{n-1},e_{n-1}))\Bigr)|t|.
\end{multline*}
as $\nu_n<\sigma_{n-1}/2$.
Thus we proved the second condition of \eqref{eqlast2} for
$|t|<8\Delta/\nu_n$.

Now we consider the case $|t| \geq 8\Delta/\nu_n$. From $d(x',\limx) \leq \delta$ and~\eqref{hahamissed} we have
$$\|\pi x'-\pi \limx\| \leq \Delta \leq \nu_n|t|/8$$
so we get, using $(\limx,\lime) \in G_{p_n}(x_{n-1},e_{n-1},\sigma_{n-1}-\nu_{n}/2)$ from
\eqref{eq.xinfg},
\begin{multline*}
|(f(\pi x'+te_{n-1})-f(\pi x'))-(f(\pi x_{n-1}+te_{n-1})-f(\pi x_{n-1}))|\\
 \leq|(f(\pi \limx+te_{n-1})-f(\pi \limx))-(f(\pi x_{n-1}+te_{n-1})-f(\pi x_{n-1}))| +2\|\pi x'-\pi \limx\|\\
\leq \left(\sigma_{n-1}-\nu_n/2 + \Omega(w_n(\limx,\lime)-w_n(x_{n-1},e_{n-1}))\right)|t| + \nu_{n}|t|/4\\
 \leq \left(\sigma_{n-1}-\nu_n/4 + \Omega(w_n(x',e')-w_n(x_{n-1},e_{n-1}))\right)|t|,
\end{multline*}
where, in the final line, we have used $w_n(x',e') \geq w_n(\limx,\lime)$ from
\eqref{eq.wnwx'}.

This finishes the proof of $(x',e')\in D_n$ for every $n\ge1$.
Recall the property of the pair $(x_n,e_n)\in D_n$ is such that
$$
w_n(x,e)\le w_n(x_n,e_n)+\e_n
$$
for all $(x,e)\in D_n$.
Notice that by \eqref{eq.wnincr}
the right hand side of this inequality is less than or equal
to  $w_\infty(\limx,\lime)+\e_n$, thus together with
\eqref{choicex'} and \eqref{eq.wnwinf} we finally get
$$
w_\infty(\limx,\lime)+\e
\le
w_\infty(x',e')
\le
w_n(x',e')+t_n^2
\le
w_\infty(\limx,\lime)+\e_n+t_n^2.
$$
This is a contradiction as $\e>\e_n+t_n^2$.
This means that the assumption \eqref{choicex'} is false,
completing the proof of the statement of the present section.

\subsection{Proof of Theorem~\ref{th.incr}}
We first quote \cite[Lemma 4.3]{P} for determining the \frechet{} differentiability of the norm $p_{\infty}$:
\begin{lemma*}
If the norm of a Banach space $Y$ is \frechet{} differentiable on $Y \setminus \{0\}$,  $e_m \in Y$ and $t_m \geq 0$ with $\sum t_m^2 < \infty$, then the function $p \colon Y \to \R$ defined by the formula
$$p(y) := \sqrt{\|y\|^2 + \sum_{m=1}^{\infty} t_m^2 \|y-\R e_m\|^2}$$
is an equivalent norm on $Y$ that is \frechet{} differentiable on $Y \setminus \{0\}$.
\end{lemma*}
We verify the conclusions of Theorem~\ref{th.incr} for Lipschitz
function $f$ defined in \eqref{eq.deff-g}
and the norm
$\|\cdot\|'=p_\infty$.

The items (1) and (2) of Theorem~\ref{th.incr} follow from
\eqref{eq.deff-g} and \eqref{eq.normineq} as $p_n\to p_\infty$.
The ``moreover'' statement in Theorem~\ref{th.incr}
is a direct consequence of the Lemma
quoted above and the definition of $p_\infty$ in Section~\ref{sec.diff}.

For part (3) of Theorem~\ref{th.incr} we define $\tilde x = \limx$ and
$\tilde e = \lime/\|\lime\|'$.
Then we have $(\tilde x,\tilde e)\in D$
and $\|\tilde e\|' = 1$.
Further we have
$$
f'(\pi\tilde x,\tilde e) = w_{\infty}(\limx,\lime) \geq w_{0}(x_0,e_0)
= f'(\pi x_0,e_0)
$$
by Definition~\ref{defnweight} and \eqref{eq.wnincr}.
Now given any $\e > 0$ we choose $\delta > 0$ as in
Section~\ref{localmaximality} and then define the open neighbourhood  of
$\tilde x$ in $M$ by $N_{\e}= \ball{\tilde x}{\delta}$.

Subsequently if~\eqref{eqincr} is satisfied then as $2\Theta \leq \Omega$, from Lemma~\ref{omegalemma}(1), and further
$\|\lime\|' \leq 2 \|\lime\| = 2$,
we have
that $f'(\pi\tilde x,\tilde e)\le f'(\pi x',e')$ implies
$$
(x',e') \in G_{p_{\infty}}(\limx,\lime,0).
$$
Then we have just showed in Section~\ref{localmaximality}, as $x' \in \ball{x}{\delta}$,
$$w_{\infty}(x',e') < w_{\infty}(\limx,\lime) +\e$$
by \eqref{eq.maxw},
and so replacing $w_{\infty}(x',e')$ by $f'(\pi x',e')$ we conclude
\begin{equation*}
f'(\pi x',e') < f'(\pi\tilde  x,\tilde e)+\e.
\end{equation*}
\qed
\section{Set theory}
\label{Set_theory}

In this section we shall prove Theorem~\ref{th.passtoclosed}.
We recall its hypotheses:
$(Y,d)$ is a metric space and $(K_r)_{r \in R}$ is a collection of non-empty compact subsets of $Y$ indexed by $(R,\gamma)$, a non-empty metric space, such that
\begin{equation}
\label{eq.metricproperty}
\gamma(r,s) \leq \mathcal{H}(K_r,K_s)
\end{equation}
for every $r,s \in R$, where $\mathcal{H}$ denotes the Hausdorff distance.

Further $\G$ is a $G_{\delta}$ subset of $Y$ containing every element of the family $(K_r)_{r \in R'}$ where $R' \subseteq R$ is $\gamma$-dense
and $r_0\in R'$.
We further recall that $\rho,\e_0 >0$ are such that for every $\e \in (0,\e_0)$ there exists $R(\e) \subseteq R$ such that \eqref{eq.repsilon} holds:
\begin{equation}
\notag
\begin{tabular}{l}
\textbullet\quad for every $s \in R$ there exists $t \in R(\e)$ with $\gamma(t,s) < \e$,\\
\textbullet \quad
for every subset $S$ of $Y$ of diameter at most $\rho \e$
the set\\ \quad\ $\{r \in R(\e): S\cap K_r\ne\emptyset\}$ is finite.
\end{tabular}
\end{equation}

We may assume $\rho \in (0,1)$ is fixed. Write $\G = \bigcap_{n=1}^{\infty}\G_n$ where $(\G_n)$
is a nested sequence of
open subsets of $Y$, $\G_{n+1} \subseteq \G_{n}$ for each $n \geq 1$.

We first observe that due to the fact that $\G$ contains
a $\gamma$-dense collection of compacts $K_r$,
we may replace the families $R(\e)$
of compacts with families $R'(\e) \subseteq R'$, so that
$K_r \subseteq \G$ for every $r \in R'(\e)$, and properties listed in
the following lemma are satisfied.
\begin{lemma}
For every $\e \in (0,\e_0)$ we can find $R'(\e) \subseteq R$ such that $K_r \subseteq \G$ for all $r \in R'(\e)$ and
\begin{itemize}
\item
for every $r\in R$ there exists $t \in R'(\e)$ with $\gamma(t,r) < \e$,
\item
for every subset $B$ of $Y$ of
diameter at most $\frac{4}{5}\rho \e$
the set
\begin{equation}
\label{firstfin}
F^B(\e) := \{t \in R'(\e) \text{ with }K_t \cap B \neq \emptyset \}
\end{equation}
is finite.
\end{itemize}
\end{lemma}

\begin{proof}
For each $s \in R$ take $t_s \in R'$ with $\gamma(t_s,s) < \rho\e/10$, using the density of $R'$. Set
$$
R'(\e) = \{t_s :
s \in R(4\e/5)\}.
$$
It is clear that $K_r \subseteq \G$ for every $r \in R'(\e)$ and that for every $r \in R$ we can find $t \in R'(\e)$ with $\gamma(t,r) < 4\e/5+\rho\e/10 < \e$.

Now if $t \in F^B(\e)$ then, writing $t = t_s$ with $s \in R(4\e/5)$, we see
 from $\gamma(t_s,s)< \rho\e/10$
and \eqref{eq.metricproperty}
that $K_{s}$ intersects $\overline{B}_{\rho\e/10}(B)$; this set has diameter at most $\rho \e$ so the set $F^B(\e)$ is finite by \eqref{eq.repsilon}.
\end{proof}

We now define the set
\begin{equation}
\label{defT}
\mT = \{(r,w,\alpha) \in R \times (0,\e_0) \times (0,\infty) \text{ such that }K_r \subseteq \G \text{ and }w \leq \alpha\}.
\end{equation}
Here $w\in(0,\e_0)$ denotes the width of the neighbourhood
$T=\ballcl{K_r}{w}$ around $K_r$;
as we mentioned earlier in Remark~\ref{rem.wedge} our main example is
the case when $K_r$ is a
wedge,
then $T$ is an angled tube around
$K_r$. A slightly bigger neighbourhood $\ballcl{K_r}{\alpha}$,
defined by the third parameter, is considered as a neighbourhood of the
tube $T$ just constructed, in which we plan to choose smaller tubes that
approximate $T$. Therefore each element $(r,w,\alpha)\in\mT$ presents
a tube $\ballcl{K_r}{w}$
with some ``safe'' neighbourhood $\ballcl{K_r}{\alpha}$.
For convenience, we will use these terms
even in the general case when $K_r$ are arbitrary compacts;
we will also refer
to elements $(r,w,\alpha)\in\mT$ as tube triples.

For fixed $r_0\in R'$ choose  $0<w_0<\alpha_0<\e_0$ so that
\begin{equation}
\label{initialR}
R_0 = \{(r_0,w_0,\alpha_0)\}\subseteq\mT.
\end{equation}
We shall now construct, for each $k \geq 1$, a set $R_k \subseteq \mT$ inductively by adding, for every $(r,w,\alpha) \in R_l$ where $l < k$, a collection $R_{k,l} = R_{k,l}(r,w,\alpha)$ of tube triples $(t,v,\beta)\in\mT$ with $\ballcl{K_t}{v} \subseteq \G_k$ such that the collection $(K_t)_{(t,v,\beta) \in R_{k,l}}$ well approximates the collection of all compacts $(K_s)_{s \in R}$ when restricted to
the ``safe'' neighbourhood $\ballcl{K_r}{\alpha}$.
First let
\begin{equation}
\label{smallconstants}
r_{k,l} \in (0,\rho/10)
\end{equation}
for each $0 \leq l < k$,
where $\rho\in(0,1)$ is the number fixed in the beginning of the present section. Later,
in \eqref{newstipulation}, we will impose additional restrictions on
$(r_{k,l})$; however
Lemmas~\ref{fromltok}, \ref{finite} and \ref{towardsclosedness}
we prove up to that point
are valid for any $r_{k,l} \in (0,\rho/10)$.

\begin{lemma}
\label{fromltok}
If $0 \leq l < k$ and $(r,w,\alpha) \in \mT$ then there is a set
$$R_{k,l} = R_{k,l}(r,w,\alpha) \subseteq \mT$$
such that
\begin{enumerate}
\item[(1)]
for every $s \in R$ with $K_{s} \subseteq \overline{B}_{\alpha}(K_r)$ there exists $(t,v,\beta) \in R_{k,l}$ such that
$$\gamma(t,s) \leq \frac{10}{\rho}r_{k,l}w,$$
\item[(2)]
if $(t,v,\beta) \in R_{k,l}$ then $\beta = r_{k,l}w  < \alpha/10$ and $v < \e_0/k$,
\item[(3)]
if $(t,v,\beta) \in R_{k,l}$ then $\overline{B}_{v}(K_{t}) \subseteq \G_k$ and $K_{t} \subseteq \overline{B}_{2\alpha}(K_r)$,
\item[(4)]
if $B \subseteq Y$ has diameter at most $8r_{k,l}w$ then the set
\begin{equation*}
F = F_{k,l}^B(r,w,\alpha)
\end{equation*}
of all $(t,v,\beta) \in R_{k,l}$ such that $K_t$ intersects $B$, is finite,
\item[(5)]
there exists $v > 0$ such that $(r,v,r_{k,l}w) \in R_{k,l}$.
\end{enumerate}
\end{lemma}

\begin{proof}
For each $t \in R$ with $K_t \subseteq \G$ we can pick $v_t \in (0,\e_0/k)$ such that $v_t \leq r_{k,l}w$ and
$$\overline{B}_{v_t}(K_{t}) \subseteq \G_k,$$
as $K_{t} \subseteq \G \subseteq \G_k$, $K_t$ is compact and $\G_k$ is open. Now let
$$
\e = \frac{10}{\rho}r_{k,l}w.
$$

Note that $\e < w < \e_0$ from~\eqref{smallconstants} and~\eqref{defT} and that for any $t \in R'(\e) \cup \{r\}$ we have $K_t \subseteq \G$. So we may set
\begin{equation*}
R_{k,l} = \{(t,v_t,r_{k,l}w) : t \in R'(\e) \cup \{r\}\text{ is such that }
K_t \subseteq \overline{B}_{2\alpha}(K_r)\}.
\end{equation*}
Observe that $R_{k,l} \subseteq \mT$, using the definition of $v_t$.

To see item
(1) of the lemma, for $s \in R$ with $K_s \subseteq \overline{B}_{\alpha}(K_r)$ we may pick $t \in R'(\e)$ with $\gamma(t,s) < \e$. Then $\gamma(t,s) \leq w \leq \alpha$ so that $K_t \subseteq \overline{B}_{\alpha}(K_s)$ using~\eqref{eq.metricproperty}. It then follows that $K_t \subseteq \overline{B}_{2\alpha}(K_r)$ so that $(t,v_t,r_{k,l}w) \in R_{k,l}$.

Items (2) and (3) are immediate.

For (4) note that if $(t,v_t,r_{k,l}w) \in F$ then as $t \in R'(\e) \cup \{r\}$ and the set $B$ has diameter at most $\frac{4}{5} \rho \e$ we have
$$t \in F^B(\e) \cup \{r\};$$
see~\eqref{firstfin}. As this set is finite then so is $F$.

Finally item (5) is immediate with $v = v_r$.
\end{proof}

Recall from~\eqref{initialR} that we have defined $R_0 \subseteq \mT$. Now for $k \geq 1$ define $R_k \subseteq \mT$ by the recursion
\begin{equation}
\label{recursiveR}
R_k = \bigcup_{l=0}^{k-1}\bigcup_{(r,w,\alpha) \in R_l} R_{k,l}(r,w,\alpha).
\end{equation}
Note that for any $(t,v,\beta) \in R_k$ we have
\begin{equation}
\label{moreincs}
K_t \subseteq \G \text{ and }\overline{B}_{v}(K_{t}) \subseteq \G_k
\end{equation}
and
\begin{equation}
\label{worthpointingout}
0 < v \leq \min \left(\beta,\frac{\e_0}{k}\right)
\end{equation}
using~\eqref{defT} and Lemma~\ref{fromltok}, (2) and (3).

Next lemma proves that the collection of tube triples $R_k$
has some local finiteness in its structure; we will use this property later
to prove that if we consider unions of all tubes on each level and
then intersect these unions up to a certain level then the resulting set
is closed, see Definition~\ref{defMinf} and Lemma~\ref{towardsclosedness}.
\begin{lemma}
\label{finite}
If $y \in Y$ and $k \geq 0$ there exists $\delta_k = \delta_k(y) > 0$ such that
the set
\begin{equation*}
F_k=F_k(y) := \{(r,w,\alpha) \in R_k \text{ such that } d(y,K_r) \leq \delta_k + 3\alpha\}
\end{equation*}
is finite.
\end{lemma}

\begin{proof}
Let $y \in Y$. For any $\delta_0 > 0$ we pick, the set $F_0 \subseteq R_0$ will be finite. Suppose now that $k \geq 1$ and we have picked $\delta_l > 0$ for every $0 \leq l < k$ such that $F_l$ is finite.

Pick $\delta_k > 0$ such that for every $l < k$ we have $\delta_k < \delta_l$ and, for any $(r,w,\alpha) \in F_l$, $\delta_k < r_{k,l}w$.
We shall show that $F_k$ is finite.

Suppose that $(t,v,\beta) \in F_k$. Note that $(t,v,\beta) \in R_{k,l}(r,w,\alpha)$ where $l < k$ and $(r,w,\alpha) \in R_l$, using~\eqref{recursiveR}. Note that $K_{t} \subseteq \overline{B}_{2\alpha}(K_r)$ by Lemma~\ref{fromltok}(3).
Hence
\begin{align*}
d(y,K_r) &\leq d(y,K_t) + 2\alpha\\
 &\leq \delta_k + 3\beta + 2\alpha\\
 &\leq \delta_l + 3\alpha
\end{align*}
using $\delta_k < \delta_l$ and $\beta =r_{k,l} w < \alpha/10$ from Lemma~\ref{fromltok}(2).
Hence $(r,w,\alpha) \in F_l$ and so $\delta_k < r_{k,l}w$. We get $d(y,K_t) \leq \delta_k + 3\beta < 4r_{k,l}w$ so that
$$K_{t} \cap \overline{B}_{4r_{k,l}w}(y) \neq \emptyset$$
and $(t,v,\beta) \in F_{k,l}^{\overline{B}_{4r_{k,l}w}(y)}(r,w,\alpha)$; see Lemma~\ref{fromltok}(4).

We conclude that
\begin{equation*}
F_{k} \subseteq \bigcup_{l=0}^{k-1}\bigcup_{(r,w,\alpha) \in F_l}F_{k,l}^{\overline{B}_{4r_{k,l}w}(y)}(r,w,\alpha),
\end{equation*}
which is finite by Lemma~\ref{fromltok}(4).

\end{proof}

\begin{definition}
\label{defMinf}
If $k \geq 1$, $\lambda \in [0,1]$ and $w > 0$ we define $M_{k}(\lambda,w)$ to be the set of $y \in Y$ such that there exist integers $n \geq 1$,
$0 = l_0 < l_1 < \dots < l_n = k$
and tube triples $(r_m,w_m,\alpha_m) \in R_{l_m}$ for $0 \leq m \leq n$ with
\begin{enumerate}
\item[(1)]
$(r_{m},w_{m},\alpha_{m}) \in R_{l_{m},l_{m-1}}(r_{m-1},w_{m-1},\alpha_{m-1})$ for $1 \leq m \leq n$,
\item[(2)]
$d(y,K_{r_m}) \leq \lambda \alpha_m$ for $0 \leq m \leq n$,
\item[(3)]
$d(y,K_{r_n}) \leq \lambda w_n$,
\item[(4)]
$w_n = w$.
\end{enumerate}

We then let
\begin{equation*}
M_{k}(\lambda) = \bigcup_{w > 0}M_{k}(\lambda,w).
\end{equation*}
\end{definition}

\begin{remark*}
Note that Definition~\ref{defMinf}(3)
implies that
$M_k(\lambda)$ is a subset of  the union
$\bigcup \ballcl{K_r}{\lambda w}$, where the union is taken
over the collection of all
tube triples
$(r,w,\alpha)\in R_k$. Since each of those tubes is inside $\G_k$
by \eqref{moreincs},  we conclude
$M_{k}(\lambda) \subseteq \G_k$.
Further from \eqref{initialR}, \eqref{recursiveR}, Lemma~\ref{fromltok}(5), Definition~\ref{defMinf}(2) and \eqref{moreincs},
\begin{equation}
\label{r0alwaysin}
K_{r_0} \subseteq M_k(\lambda)
\end{equation}
for all $k \geq 1$ and $\lambda \in [0,1]$. Finally if $M_{k}(\lambda,w) \neq \emptyset$ then by Lemma~\ref{fromltok}(2),
\begin{equation}
\label{boundonw}
w < \e_0/k.
\end{equation}
\end{remark*}
\begin{lemma}
\label{towardsclosedness}
For any $k \geq 1$ and $\lambda \in [0,1]$, the set $M_k(\lambda)$
is a closed subset of $(Y,d)$.
\end{lemma}

\begin{proof}
Suppose that $\E{y}{i} \in M_k(\lambda)$ with $\E{y}{i} \rightarrow y \in Y$. It suffices to show that $y \in M_k(\lambda)$.

For each $i \geq 1$ we have $\E{y}{i} \in M_k(\lambda)$; therefore we can find $\E{n}{i} \geq 1$,
$$0 = \E{l_0}{i} < \dots < \E{l_{\E{n}{i}}}{i} = k$$
and
$$(\E{r_m}{i},\E{w_m}{i},\E{\alpha_m}{i}) \in R_{\E{l_m}{i}}$$
for
$0 \leq m \leq \E{n}{i}$ such that the conditions in
Definition~\ref{defMinf}(1)---(3)
are satisfied:
\begin{align}
\label{dunnoone}
&\text{$(\E{r}{i}_{m},\E{w}{i}_{m},\E{\alpha}{i}_{m}) \in R_{\E{l}{i}_{m},\E{l}{i}_{m-1}}(\E{r}{i}_{m-1},\E{w}{i}_{m-1},\E{\alpha}{i}_{m-1})$ for $1 \leq m \leq \E{n}{i}$}\\
\label{dunnotwo}
&\text{$d(\E{y}{i},K_{\E{r}{i}_m}) \leq \lambda \E{\alpha}{i}_m$ for $0 \leq m \leq \E{n}{i}$}\\
\label{dunnothree}
&\text{$d(\E{y}{i},K_{\E{r}{i}_{\E{n}{i}}}) \leq \lambda \E{w}{i}_{\E{n}{i}}$.}
\end{align}
As $1 \leq \E{n}{i} \leq k$ we may assume, passing to a subsequence if necessary, that $\E{n}{i} = n$ is constant. But then as $0 \leq \E{l_m}{i} \leq k$ we may assume, passing to another subsequence, that $\E{l_m}{i} = l_m$ is constant for each $0 \leq m \leq n$ with $0 = l_0 < l_1 < \dots < l_n = k$.

Fixing $m$ then as $d(y,\E{y}{i}) \rightarrow 0$, $\lambda \leq 1$ and
\begin{equation*}
d(y,K_{\E{r}{i}_m}) \leq d(y,\E{y}{i}) + \lambda \E{\alpha}{i}_{m},
\end{equation*}
from~\eqref{dunnotwo}, we have $(\E{r_m}{i},\E{w_m}{i},\E{\alpha_m}{i}) \in F_{l_m}(y)$ for $i$ sufficiently high; see Lemma~\ref{finite}. As this set is finite we can assume, passing to another subsequence, that
\begin{equation*}
(\E{r_m}{i},\E{w_m}{i},\E{\alpha_m}{i}) = (r_m,w_m,\alpha_m)
\end{equation*}
is constant for each $0 \leq m \leq n$, with $(r_m,w_m,\alpha_m) \in R_{l_m}$. Further from \eqref{dunnoone}---\eqref{dunnothree} we have
\begin{itemize}
\item
$(r_{m},w_{m},\alpha_{m}) \in R_{l_{m},l_{m-1}}(r_{m-1},w_{m-1},\alpha_{m-1})$ for $1 \leq m \leq n$
\item
$d(\E{y}{i},K_{r_m}) \leq \lambda \alpha_m$ for $0 \leq m \leq n$
\item
$d(\E{y}{i},K_{r_n}) \leq \lambda w_n$.
\end{itemize}
Taking the $i \rightarrow \infty$ limit and using
$\E{y}{i} \rightarrow y$ we obtain
\begin{itemize}
\item
$d(y,K_{r_m}) \leq \lambda \alpha_m$ for $0 \leq m \leq n$
\item
$d(y,K_{r_n}) \leq \lambda w_n$,
\end{itemize}
so that $y \in M_{k}(\lambda)$.
\end{proof}

Up to this point we have let $r_{k,l} \in (0,\rho/10)$ be arbitrary; see~\eqref{smallconstants}. We now further stipulate that if $0 \leq l < l' \leq k$ then we have
\begin{equation}
\label{newstipulation}
r_{k+1,k} \leq \frac{1}{k} \qquad\text{and}\qquad r_{k+1,l} \leq \frac{1}{k}r_{l',l}.
\end{equation}

We now come to the crucial lemma. It proves that if we consider a point $y$ is
in $M_k(\lambda,w)$ and $\lambda'>\lambda$, then the whole
$(\lambda'-\lambda)w$-neighbourhood of $y$ is inside $M_k(\lambda',w)$.
If, however, we want to find compacts $K_t$ close to $y$ of bigger size,
$\delta>(\lambda'-\lambda)w/2$, we can accomplish this as long as we agree to consider
tube sets constructed on subsequent levels.
\begin{lemma}
\label{thehardwork}
Suppose $k \geq 1$, $0 \leq \lambda < \lambda+\psi \leq 1$, $w > 0$,
$\e\in(0,1)$ and $y \in M_k(\lambda,w)$. Then
\begin{enumerate}
\item[(1)]
$\overline{B}_{\psi w}(y) \subseteq M_k(\lambda+\psi,w)$,
\item[(2)]
if $2\delta \in (\psi w,\psi \alpha_0)$ and
$20/(\rho \psi k)<\e<1$
then for each $s \in R$ with $K_s \subseteq \overline{B}_{\delta}(y)$ there exists $t \in R$ with $\gamma(t,s) < \e \delta$ and
$K_t \subseteq M_{k+j}(\lambda+\psi)$ for all $j \geq 1$.
\end{enumerate}
\end{lemma}

\begin{proof}
From Definition~\ref{defMinf} we can find integers $n \geq 1$,
\begin{equation*}
0 = l_0 < l_1 < \dots < l_n = k
\end{equation*}
and tube triples
$(r_m,w_m,\alpha_m) \in R_{l_m}$ for $0 \leq m \leq n$ with
\begin{align}
\label{aginone}
&\text{$(r_{m},w_{m},\alpha_{m}) \in R_{l_{m},l_{m-1}}(r_{m-1},w_{m-1},\alpha_{m-1})$ for $1 \leq m \leq n$}\\
\label{agintwo}
&\text{$d(y,K_{r_m}) \leq \lambda \alpha_m$ for $0 \leq m \leq n$}\\
\label{aginthree}
&\text{$d(y,K_{r_n}) \leq \lambda w_n$}\\
\label{aginfour}
&\text{$w_n = w$.}
\end{align}
Note that
\begin{equation}
\label{eq.alineq}
\alpha_m = r_{l_{m},l_{m-1}}w_{m-1} < \alpha_{m-1}
\end{equation}
for each $1 \leq m \leq n$ by Lemma~\ref{fromltok}(2).

To establish (1)
of the present lemma, suppose $d(y',y) \leq \psi w$; then from \eqref{agintwo} and \eqref{aginthree},
\begin{align*}
&\text{$d(y',K_{r_m}) \leq \lambda \alpha_m + \psi w$ for $0 \leq m \leq n$}\\
&\text{$d(y',K_{r_n}) \leq \lambda w_n + \psi w$.}
\end{align*}
Using~\eqref{worthpointingout},~\eqref{aginfour} and~\eqref{eq.alineq} we have $w = w_n \leq \alpha_n \leq \alpha_m$ so that
\begin{align*}
&\text{$d(y',K_{r_m}) \leq (\lambda+\psi) \alpha_m$ for $0 \leq m \leq n$}\\
&\text{$d(y',K_{r_n}) \leq (\lambda+\psi) w_n$;}
\end{align*}
combining these with~\eqref{aginone} and~\eqref{aginfour} we get $y' \in M_{k}(\lambda+\psi,w)$, as required.

We now turn to (2).
We claim that we can find $m$ with $0 \leq m \leq n$ and
\begin{equation}
\label{crittriple}
(t,w,\alpha) \in R_{k+1,l_{m}}(r_m,w_m,\alpha_m)
\end{equation}
where $2\delta \leq \psi \alpha_m$ and $\gamma(t,s) < \e \delta$.

To see this suffices, note first that as $\mathcal{H}(K_t,K_s) \leq \gamma(t,s) < \delta$, using $\e \leq 1$, we have
\begin{equation}
\label{incs}
K_{t} \subseteq \overline{B}_{\delta}(K_s) \subseteq \overline{B}_{2\delta}(y) \subseteq \overline{B}_{\psi \alpha_m}(y),
\end{equation}
where we have also used $2\delta \leq \psi \alpha_m$ from the claim to be proved
and $K_s \subseteq \ballcl{y}{\delta}$ from the
hypothesis of (2).

Now let $l'_{j} = l_{j}$ and $(r'_j,w'_j,\alpha'_j) = (r_j,w_j,\alpha_j)$ for $j \leq m$ and $l_{m+j}' = k+j$ for $j \geq 1$ and, using~\eqref{crittriple}
for $(r_{m+1}',w_{m+1}',\alpha_{m+1}')$ and Lemma~\ref{fromltok}(5), pick inductively
$$(r_{m+j}',w_{m+j}',\alpha_{m+j}') \in R_{l_{m+j}',l_{m+j-1}'}(r_{m+j-1}',w_{m+j-1}',\alpha_{m+j-1}')$$
for each $j \geq 1$, with $r'_{m+j} = t$. Then for any $y' \in K_{t}$, as
$$
d(y',K_{r_j}) \leq d(y,K_{r_{j}})+\psi \alpha_m \leq (\lambda+\psi)\alpha'_j
$$
for $j \leq m$, using~\eqref{agintwo} and~\eqref{incs}, while $d(y',K_{r_{m+j}'}) = 0$ for $j \geq 1$
from $y'\in K_t=K_{r_{m+j}'}$, we have $y' \in M_{k+j}(\lambda+\psi)$ for $j\geq 1$ as required.

We now establish the claim. Suppose first that $2\delta \leq \psi \alpha_n$. Then as
$$K_s \subseteq \overline{B}_{\delta}(y) \subseteq \overline{B}_{\lambda \alpha_n + \delta}(K_{r_n}) \subseteq \overline{B}_{\alpha_n}(K_{r_n}),$$
using~\eqref{agintwo}, we may pick, by Lemma~\ref{fromltok}(1),
$(t,w,\alpha)\ \in R_{k+1,k}(r_n,w_n,\alpha_n)$ with
$$\gamma(t,s) \leq \frac{10}{\rho}r_{k+1,k}w_n \leq \frac{10}{\rho}\frac{1}{k} \frac{2\delta}{\psi} < \e \delta$$
using~\eqref{newstipulation} and $2\delta \in (\psi w_n,\psi)$. Thus we can satisfy the claim with $m = n$.

Suppose instead that $\psi \alpha_n < 2\delta$. As $2\delta \leq \psi \alpha_0$ we can find $m$ with
\begin{equation}
\label{criticalancestor}
\psi \alpha_{m+1} < 2\delta \leq \psi \alpha_{m}
\end{equation}
where $0 \leq m \leq n-1$.
Then as
$$K_s \subseteq \overline{B}_{\delta}(y) \subseteq \overline{B}_{\lambda \alpha_m + \delta}(K_{r_m}) \subseteq \overline{B}_{\alpha_m}(K_{r_m}),$$
we may pick, by Lemma~\ref{fromltok}(1),
$(t,w,\alpha) \in R_{k+1,l_m}(r_m,w_m,\alpha_m)$ with
\begin{equation*}
\gamma(t,s) \leq \frac{10}{\rho}r_{k+1,l_m}w_m \leq \frac{10}{\rho}\frac{1}{k} r_{l_{m+1},l_m}w_m = \frac{10}{\rho}\frac{1}{k} \alpha_{m+1} < \frac{10}{\rho}\frac{1}{k} \frac{2\delta}{\psi} < \e \delta
\end{equation*}
using~\eqref{newstipulation}
with $l_m<l_{m+1}\le k$,~\eqref{eq.alineq} and~\eqref{criticalancestor}. Thus the claim is satisfied.
\end{proof}

\noindent

\subsection{Proof of Theorem~\ref{th.passtoclosed}}

We are now ready to prove Theorem~\ref{th.passtoclosed}.

Assume $r_0$ used in \eqref{initialR} is the one given by hypothesis of Theorem~\ref{th.passtoclosed}.

Given $\lambda \in [0,1]$ we set
$$
\T_{\lambda} = \bigcap_{k=1}^{\infty} J_k(\lambda),
\quad\text{where}\quad
J_{k}(\lambda) = \bigcup_{k \leq n \leq (1+\lambda)k}M_n(\lambda).
$$
Note that as \eqref{r0alwaysin}  implies
$K_{r_0}\subseteq M_n(\lambda) \subseteq \G_n \subseteq \G_k$ for
$n \geq k$, we have $K_{r_0}\subseteq J_{k}(\lambda) \subseteq \G_k$ for every
$k \geq 1$ and hence
$K_{r_0}\subseteq\T_{\lambda} \subseteq \G$ for every $\lambda\in[0,1]$. Similarly as $M_k(\lambda)$ is closed by Lemma~\ref{towardsclosedness}, the set $J_{k}(\lambda)$ is also closed for every $k \geq 1$, and hence $\T_{\lambda}$ is closed for every $\lambda \in [0,1]$. We further note that if $0 \leq \lambda_1\le \lambda_2 \leq 1$ then as $M_k(\lambda_1) \subseteq M_k(\lambda_2)$ from
Definition~\ref{defMinf}, we have $J_k(\lambda_1) \subseteq J_k(\lambda_2)$ and hence we have $\T_{\lambda_1} \subseteq \T_{\lambda_2}$.

Assume $\eta>0$, $0\le\lambda'<\lambda\le1$ and $y\in\T_{\lambda'}$.
By the definitions of $\T_{\lambda'}$ and $M_{k}(\lambda')$ and the last part of Definition~\ref{defMinf}, there exists,
for each $k \geq 1$, an index $n_k$ with $k \leq n_k \leq (1+\lambda')k$ and
$w_k > 0$ such that
$y \in M_{n_k}(\lambda',w_k)$. Let $\psi = \lambda-\lambda' > 0$.

Pick $\delta_1 > 0$ with $2\delta_1 < \psi w_k$ for every
$k \leq 20/(\rho \psi \eta)$, where $\rho\in(0,1)$
is the number fixed in the beginning of the present section.
 Now suppose that
$\delta \in (0,\delta_1)$.
We need to show that if $K_s\subseteq\ballcl{y}{\delta}$
for some $s\in R$ then there exists $t\in R$ such that $K_t\subseteq T_\lambda$
and $\gamma(t,s)<\eta\delta$.
 Let  $k_0 \geq 1$ be the minimal index $k$ such that
$2\delta > \psi w_k$. Such $k_0$ exists as \eqref{boundonw} implies
$w_k\to0$.
Note that $k_0 > 20/(\rho \psi \eta)$.
In particular  $\psi k_0>1$ and so
$$k_0<n_{k_0}+1<(1+\lambda')k_0+\psi k_0=(1+\lambda)k_0.$$

By Lemma~\ref{thehardwork}(2) there exists $t\in R$ such that
$\gamma(s,t)<\eta\delta$ and
$K_t \subseteq M_j(\lambda)$ for every $j \geq n_{k_0}+1$,
so that $K_t\subseteq J_k(\lambda)$ for all $k\ge k_0$.
Note that $\gamma(t,s)<\eta\delta<\delta$ implies that
$K_t\subseteq\ballcl{y}{2\delta}\subseteq	\ballcl{y}{\psi w_k}$
for every $k<k_0$. By Lemma~\ref{thehardwork}(1)
we conclude $K_t\subseteq M_{n_k}(\lambda,w_k)$ for every $k<k_0$.

Hence $K_t \subseteq T_{\lambda}$ as required.

\qed

\bib
\end{document}